\documentclass[11pt]{article}
\usepackage{amssymb,amsmath,amsthm,multirow,color,xcolor,cancel,bbm,comment,mathrsfs}
\usepackage[textheight=22cm,textwidth=16cm]{geometry}

\usepackage{calc}
\usepackage{tikz}
\usepackage{hyperref}
\usetikzlibrary{decorations.pathreplacing,decorations.markings,automata,arrows,shadows,patterns,calc}
\newlength{\rnodo}
\tikzstyle{main node}=[outer sep=1,inner sep=0,circle,thick,draw,minimum size=2\rnodo,fill=black!10]
\tikzstyle{and node}=[outer sep=1,inner sep=0,circle,thick,draw,minimum size=2\rnodo,fill=black!60,text=white]
\tikzstyle{or node}=[outer sep=1,inner sep=0,circle,thick,draw,minimum size=2\rnodo]
\tikzstyle{ao node}=[outer sep=1,inner sep=0,circle,thick,draw,minimum size=2\rnodo,fill=black!50,text=white]
\tikzstyle{labels}=[inner sep=0pt,font=\scriptsize,auto,circle]
\tikzstyle{arcos}=[-latex,thick]
\tikzstyle{arcossym}=[latex-latex,thick]
\newcommand{\Loop}[4][]{\draw[arcos, #1](#2.#3+10)..controls +(#3+30:2\rnodo) and +(#3-30:2\rnodo).. node[labels]{#4}(#2.#3-10)}

\tikzstyle{vnode}=[circle,fill,minimum size=4pt,outer sep=1pt,inner sep=0pt]
\tikzstyle{inode}=[circle,minimum size=3pt,outer sep=1pt,inner sep=0pt,font=\footnotesize]
\tikzstyle{positive}=[green!70!black,thick,-to]
\tikzstyle{negative}=[red,thick,-to]
\tikzstyle{dpos}=[green!70!black,thick]
\tikzstyle{dneg}=[red,thick]
\tikzstyle{claro}=[fill=black!10]
\tikzstyle{medio}=[fill=black!30]
\tikzstyle{oscuro}=[fill=black!50]

\tikzstyle{ssigno}=[thick]
\tikzstyle{grueso}=[line width=1.5pt]
\tikzstyle{texto}=[font=\small,inner sep = 2pt]
\tikzstyle{cuadros}=[thick, rounded corners]
\newlength{\sombra}
\newlength{\ssombra}
\newlength{\unit}
\newlength{\munit}
\newlength{\punit}
\newlength{\sep}
\newcommand{\flecha}[3][]{\draw[very thick,-to] #2 to[#1] +(#3,0);}

\newlength{\amplit}
\newlength{\grosor}
\newlength{\segamp}
\newcommand{\llaveD}[3][]{
	\setlength{\amplit}{5pt}
	\setlength{\grosor}{0.75pt}
	#1
	\setlength{\segamp}{0.5\amplit}
	\addtolength{\segamp}{-\grosor}
	\def\largo{0.5\amplit,0.5*tan(5)*\amplit}
	\coordinate (c1) at #2;
	\coordinate (c7) at #3;
	\coordinate[rotate around={-85:(c1)}] (c2) at ($(c1)!veclen(\largo)!(c7)$);
	\coordinate[rotate around={85:(c7)}] (c6) at ($(c7)!veclen(\largo)!(c1)$);
	\coordinate (cc) at ($(c1)!0.5!(c7)$);
	\coordinate[rotate around={-90:(cc)}] (c4) at ($(cc)!veclen(\amplit,0)!(c7)$);
	\coordinate[rotate around={5:(c4)}] (c3) at ($(c4)!veclen(\largo)! (cc)$);
	\coordinate[rotate around={-5:(c4)}] (c5) at ($(c4) !veclen(\largo)! (cc)$);
	\coordinate[rotate around={90:(c6)}](c8) at ($(c6)!veclen(\grosor,0)!(c5)$);
	\coordinate[rotate around={-90:(c5)}](c9) at ($(c5)!veclen(\grosor,0)!(c6)$);
	\coordinate[rotate around={90:(c3)}](c10) at ($(c3)!veclen(\grosor,0)!(c2)$);
	\coordinate[rotate around={-90:(c2)}](c11) at ($(c2)!veclen(\grosor,0)!(c3)$);
	\filldraw[very thin,line join=round,rounded corners=0.5\amplit] (c1) -- (c2) -- (c3) {[sharp corners]-- (c4)}--(c5)--(c6){[sharp corners]--(c7)}--(c8){[rounded corners=\segamp]--(c9)}{[sharp corners]--(c4)}{[rounded corners=\segamp]--(c10)}--(c11){[sharp corners]--cycle};
}
\newcommand{\llaveI}[3][]{\llaveD[#1]{#3}{#2}}

\newtheorem{theorem}{Theorem}
\newtheorem{proposition}{Proposition}
\newtheorem{lemma}{Lemma}
\newtheorem{corollary}{Corollary}

\newtheorem{remark}{Remark}
\newtheorem{example}{Example}

\newtheorem{claim}{Claim}

\newenvironment{subproof}{\begin{proof}[Proof of Claim]}{\end{proof}}

\def\B{\{0,1\}}
\def\ONE{\textrm{\boldmath$1$\unboldmath}}

\def\Fix{\mathrm{Fix}}
\def\fix{\mathrm{fix}}
\def\Mis{\mathrm{Mis}}
\def\mis{\mathrm{mis}}

\def\N{\mathcal{N}}
\def\G{\mathcal{G}}

\def\C{\mathcal{C}}
\def\H{\mathcal{H}}

\def\s{\sigma}

\title{Fixed points in conjunctive networks and maximal\\ independent sets in graph contractions}

\author{
Julio Aracena\footnote{CI$^2$MA and Departamento de Ingenier\'ia Matem\'atica, Universidad de Concepci\'on, Av. Esteban Iturra s/n, Casilla 160-C, Concepci\'on, Chile. Email: \tt{jaracena@dim.uchile.cl}}
\footnote{Partially supported by FONDECYT project 1151265, BASAL project CMM, Universidad de Chile and by Centro de Investigaci\'on en Ingenier\'ia Matem\'atica (CI$^2$MA), Universidad de Concepci\'on.}
\and 
Adrien Richard\footnote{Laboratoire I3S, UMR CNRS 7271 \& Universit\'e de Nice-Sophia Antipolis, France. Email: \tt{richard@unice.fr}}~\footnote{Corresponding author.} 
\footnote{Partially supported by CNRS project PICS06718 and University of Nice-Sophia Antipolis.} 
\and 
Lilian Salinas\footnote{Department of Computer Sciences and CI$^2$MA, University of Concepci\'on, Edmundo Larenas 215, Piso~3, Concepci\'on, Chile. Email: \tt{lilisalinas@udec.cl}}
\footnote{Partially supported by FONDECYT project 1151265.}
}

\date{July 23, 2015; Revised April 18, 2018}

\begin{document}

\maketitle

\begin{abstract}
Given a graph $G$, viewed as a loop-less symmetric digraph, we study the maximum number of fixed points in a conjunctive boolean network with $G$ as interaction graph. We prove that if $G$ has no induced $C_4$, then this quantity equals both the number of maximal independent sets in $G$ and the maximum number of maximal independent sets among all the graphs obtained from $G$ by contracting some edges. We also prove that, in the general case, it is coNP-hard to decide if one of these equalities holds, even if $G$ has a unique induced $C_4$. 

\end{abstract}

\section{Introduction}

A {\em Boolean network}  with $n$ components is a discrete dynamical system usually defined as a map 
\[
f:\B^n\to\B^n,\qquad x=(x_1,\dots,x_n)\mapsto f(x)=(f_1(x),\dots,f_n(x)).
\]
Boolean networks have many applications. In particular, they are classical models for the dynamics of gene networks \cite{K69,T73,TA90,TK01,J02}, neural networks \cite{MP43,H82,G85,GM90,GM91} and social interactions \cite{PS83, GT83}. They are also essential tools in information theory, for the binary network coding problem~\cite{R07,GR11,GRF16}. 

In many contexts, the main parameter of $f$ is its {\em interaction graph}, the digraph $G$ on $\{1,\dots,n\}$ that contains an arc from $j$ to $i$ if $f_i$ depends on $x_j$. The arcs of $G$ can also be signed by a labeling function $\sigma$, to obtain the {\em signed interaction graph $G_\sigma$}. The sign $\sigma(j,i)$ of an arc from $j$ to $i$ then indicates whether the Boolean function $f_i$ is an increasing (positive sign), decreasing (negative sign) or non-monotone (zero sign) function of $x_j$. More  formally, denoting $e_j$ the $j$th base vector, 
\[
\sigma(j,i)=
\begin{cases}
 1&\textrm{if $f_i(x)\leq f_i(x+e_j)$ for all $x\in\B^n$ with $x_j=0$},\\
-1&\textrm{if $f_i(x)\geq f_i(x+e_j)$ for all $x\in\B^n$ with $x_j=0$},\\
 0&\textrm{otherwise}.
\end{cases}
\]
 
The signed interaction graph is very commonly considered when studying gene networks, since a gene can typically either activate (positive sign) or inhibit (negative sign) another gene. Furthermore, the signed interaction graph is usually the first reliable information that biologists obtain when they study a gene network, the actual dynamics being much more difficult to determine \cite{TK01,N15}. A central problem is then to predict these dynamics according to the signed interaction graph. Among the many dynamical properties that can be studied, fixed points are of special interest since they correspond to stable states and have often specific meaning. For instance, in the context of gene networks, they correspond to stable patterns of gene expression at the basis of particular cellular processes \cite{TA90,A04}. Many works have thus been devoted to the study of fixed points. In particular, the number of fixed points has been the subject of a stream work, e.g. in \cite{R86,ADG04b,RRT08,A08,JLV10,GR11,VL12,ARS14,GRR15}. Below, we denote by $\phi(G_\sigma)$ the {\em maximum number of fixed points} in a Boolean network with $G_\sigma$ as signed interaction graph. 

\paragraph{Positive and negative cycles}
The sign of a cycle in $G_\sigma$ is {\em positive} (resp. {\em negative}) if the product of the signs of its arcs is non-negative (resp. non-positive). Positive cycles are key structures for the study of $\phi(G_\sigma)$. A fundamental result concerning this quantity, proposed by the biologist Ren\'e Thomas, is that $\phi(G_\sigma)\leq 2$ if $G_\sigma$ has only negative cycles. This was generalized in \cite{A08} into the following upper-bound, referred as the {\em positive feedback bound} in the following: for any signed digraph $G_\sigma$,
\[
\phi(G_\sigma)\leq 2^{\tau^+(G_\sigma)},
\]
where $\tau^+(G_\sigma)$ is the minimum size of a {\em positive feedback vertex set}, that is, the minimum size of a subset of vertices intersecting every positive cycle of $G_\sigma$. An immediate consequence is that $\max_{\sigma}\phi(G_\sigma)\leq 2^{\tau(G)}$ for all digraphs $G$, where $\tau(G)$ is the minimum size of a feedback vertex set of $G$. It is worst noting that this result has been proved independently in the context of network coding in information theory \cite{R07}. Actually, a central problem in this context, called the {\em binary network coding problem}, is equivalent to identify the digraphs $G$ reaching the bound, that is, such that $\max_{\sigma}\phi(G_\sigma)=2^{\tau(G)}$ \cite{R07,GR11}. 

The positive feedback bound is very perfectible. For instance, it is rather easy to find, for any $k$, a strongly connected signed digraph $G_\sigma$ with $\phi(G_\sigma)=1$ and $\tau^+(G_\sigma)\geq k$ \cite{A15}. This is not so surprising, since the positive feedback bound depends only on the structure of positive cycles, while negative cycles may have a strong influence on fixed points. Actually, such a gap, with $\phi(G_\sigma)$ bounded and $\tau(G_\sigma)$ unbounded, {\em requires} the presence of negative cycles, since one can prove the following: there exists an unbounded function $h:\mathbb{N}\to\mathbb{N}$ such that, for any strongly connected signed digraph $G_\sigma$ with only positive cycles, $\phi(G_\sigma)\geq h(\tau^+(G_\sigma))$ 
\footnote{First, if $G_\sigma$ is strongly connected and has only positive cycles, then $\tau^+(G_\sigma)=\tau(G)$, and, by \cite[Proposition 1]{MRRS13}, $\fix(G_\sigma)=\fix(G_+)$, where $G_+$ is obtained from $G_\sigma$ by making positive all the arcs. Second, according to \cite[Lemma 6]{ARS16}, we have $\nu(G)< \phi(G_+)$, where $\nu(G)$ is the maximal size of a collection of vertex-disjoint cycles in $G$. Third, by a celebrated theorem of Reed, Robertson, Seymour and Thomas \cite{RRST95}, there exists an unbounded function $h:\mathbb{N}\to\mathbb{N}$ such that $\nu(G)\geq h(\tau(G))$. We then deduce that $\phi(G_\sigma)=\phi(G_+)> \nu(G)\geq h(\tau(G))=h(\tau^+(G_\sigma))$.}.

In this situation, it is natural to focus our attention on the influence of negative cycles. A natural framework for this problem is to fix $G$ and study the variation of $\phi(G_\sigma)$ according to labeling function $\sigma$, which gives the repartition of signs on the arcs of $G$. This is, however, a difficult problem, widely open, with very few formal results (see however \cite{ARS14,GRR15}). This comes from the versatility of negative cycles: depending on the way they are connected to positive cycles, that can either be favorable or unfavorable to the presence of many fixed points. The simple example in figure \ref{fig:influence_neg} illustrates this. As another   illustration, consider $K^+_n$ (resp. $K^-_n$), the signed digraphs obtained from $K_n$, the complete loop-less symmetric digraph on $n$ vertices, by adding a positive (resp. negative) sign to each arc. It has been proved in \cite{GRR15} that, for any $n\geq 4$, 
\begin{equation}\label{eq:complete}
\phi(K^-_n)={n\choose \lfloor\frac{n}{2}\rfloor}>\frac{2^{n+1}}{n+2}\geq \phi(K^+_n).
\end{equation}
In $K^-_n$, the sign of every cycle is given by the parity of its length, and there is thus a balance between the number of positive and negative cycles. This balance allows the presence of more fixed points than in $K^+_n$, where all the cycles are positive. Actually, we may think that the balance obtained in $K^-_n$ is optimal when zero sign is forbidden, that is, any Boolean network whose signed interaction graph is obtained from $K_n$ by adding a positive or negative sign on each arc has at most $\phi(K^-_n)$ fixed points.

\begin{figure}
\centering
\begin{tabular}{ccccc}
\begin{tikzpicture}
\coordinate (c1) at (-30:1);
\coordinate (c2) at (90:1);
\coordinate (c3) at (210:1);
\foreach \x in{1,2,3}{
	\node[vnode] (\x) at (c\x) {};
}
\path[positive]
(2) edge (1)
(3) edge (2)
(1) edge[bend right =10] (3)
(3) edge[bend right =10] (1)
;
\Loop[positive]{1}{-30}{};
\Loop[positive]{2}{90}{};
\Loop[negative]{3}{210}{};
\end{tikzpicture}
&&
\begin{tikzpicture}
\foreach \x in{1,2,3}{
	\node[vnode] (\x) at (c\x) {};
}
\path[positive]
(2) edge (1)
(3) edge (2)
(1) edge[bend right =10] (3)
(3) edge[bend right =10] (1)
;
\Loop[positive]{1}{-30}{};
\Loop[positive]{2}{90}{};
\Loop[positive]{3}{210}{};
\end{tikzpicture}
&&
\begin{tikzpicture}
\foreach \x in{1,2,3}{
	\node[vnode] (\x) at (c\x) {};
}
\path[positive]
(2) edge (1)
(3) edge (2)
(1) edge[bend right =10] (3)
(3) edge[negative,bend right =10] (1)
;
\Loop[positive]{1}{-30}{};
\Loop[positive]{2}{90}{};
\Loop[positive]{3}{210}{};
\end{tikzpicture}
\\
$\phi(G_{\sigma_1})=2$ && $\phi(G_{\sigma_2})=4$ &&$\phi(G_{\sigma_3})=5$
\end{tabular}
\caption{Green arcs are positive, and red arcs are negative. This convention is used throughout the paper.}\label{fig:influence_neg}
\end{figure}
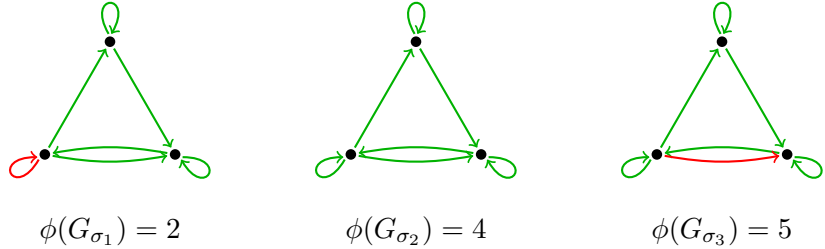

\paragraph{Results}

The variation of $\phi(G_\sigma)$ according to $\sigma$ is a difficult problem, and even the study of the repartition of signs $\sigma$ that maximizes $\phi(G_\sigma)$ is difficult. A first approach is to focus on a special class of digraphs, such as the complete digraph $K_n$ considered above. Another approach is to restrict the family of Boolean networks we consider. Here, we use both approaches simultaneously, by considering the class of {\em conjunctive networks} with {\em loop-less} and {\em symmetric} interaction graphs. 

Given a {\em full-signed} digraph $G_\sigma$, {\em i.e.} a signed digraph with no zero sign, the {\em conjunctive network} $f$ on $G_\sigma$ is defined in such a way that $f_i$ is a conjunction of positive or negative literals, where $x_j$ (resp. $\overline{x_j}$) appears as a positive (resp. negative) literal if and only if there is a positive (resp. negative) arc from $j$ to $i$. Our interest for this class of networks is twofold. Firstly, every Boolean network can be represented, up to an increase of the number of components, under the form of a conjunctive network (simply because the conjunction and the negation form a complete set of connectives). Secondly, there is a one-to-one correspondence between conjunctive networks and full-signed digraphs, which make easier the study of the relationships between structure and dynamics, because graph theoretic tools and results can be used in a more straightforward way. In particular, there are many connections with kernels and maximal independent sets \cite{ADG04b,VL12,RR13,ARS14}. 

The study of conjunctive networks in the context of gene networks has captured special interest in the last time due to increasing evidence that  the synergistic regulation of a gene by several transcription factors, corresponding to  a conjunctive function, is a common mechanism in regulatory networks \cite{nguyen2006deciphering,gummow2006reciprocal}. Besides, they have been used  in the combinatorial influence of deregulated gene sets on disease phenotype classification \cite{park2010inference} and in the construction of synthetic gene networks \cite{shis2013library}. On the other hand, the symmetry of the interaction graph of a Boolean network appears in models that describe reciprocal social relationships like acquaintance, friendship, scientific collaboration \cite{PS83, GT83}; in models of neural networks (Hopfield networks) \cite{H82,GM91} and in the modeling of intercellular signalling networks \cite{collier1996pattern,glass2016signaling}.

Since loop-less symmetric digraphs can naturally be viewed as (undirected) graph, in the following, by {\em graph} we always mean {\em loop-less symmetric digraph}. Furthermore, all the signed graphs we consider are always full-signed. We denote by $\fix(G_\sigma)$ the number of fixed points in the conjunctive network on $G_\sigma$. We denote by $G_-$ the signed digraph obtained from $G$ by adding a negative sign on every arc. We denote by $\mis(G)$ the number of maximal independent sets in $G$, and $C_k$ is the undirected cycle of length~$k$. 

A fundamental observation is that, for all graph $G$,
\begin{equation}\label{eq:fix_and_mis}
\fix(G_-)=\mis(G).
\end{equation}

Since many classes of graphs have a lot of maximal independent sets (e.g. trees), this shows that, at least for these graphs, the full-negative repartition produces many fixed points. Our main result, the following, shows that induced $C_4$ are the only obstruction to this optimality. 

\begin{theorem}\label{thm:main}
For every graph $G$, 
\[
\mis(G)\leq
\max_\sigma \fix(G_\sigma)\leq \left(\frac{3}{2}\right)^{m(G)}\mis(G)
\]
where $m(G)$ is the maximum size of a matching $M$ of $G$ such that every edge of $M$ is contained in an induced copy of $C_4$ that contains no other edge of $M$.
\end{theorem}

The upper-bound, which is by \eqref{eq:fix_and_mis} the non-trivial part, is tight: as showed below, for every $k\geq 1$ there exists a graph $G$ such that $m(G)=k$ and $\max_\sigma \fix(G_\sigma)=(\frac{3}{2})^k\mis(G)$. Obviously, if $G$ has no induced $C_4$ then $m(G)=0$ and we obtain the following corollary. 

\begin{corollary}\label{cor}
For every graph $G$ without an induced copy of $C_4$,  
\begin{equation}\label{eq:fix=mis}
\max_\sigma\fix(G_\sigma)=\mis(G).
\end{equation}
\end{corollary}

The main result of \cite{ARS14} is $\max_{G\in\mathcal \G_n}\max_{\sigma}\fix(G_\sigma)=\max_{G\in\G_n}\mis(G)$, where $\G_n$ is set of $n$-vertex graphs. The corollary above thus says something much more accurate for the smaller class of $C_4$-free $n$-vertex graphs. It is also akin to the inequality \eqref{eq:complete} discuss above, since from \eqref{eq:fix_and_mis} and \eqref{eq:fix=mis} we have $\fix(G_-)=\max_\sigma\fix(G_\sigma)$ for every $C_4$-free graph $G$. 

From a biological point of view, the relationship between the fixed points of a gene network and the maximal independent sets of its interaction graph has been exhibited, in particular,  in the modeling of Delta-Notch intercellular signalling network  \cite{collier1996pattern,glass2016signaling,afek2011biological}. The interaction graph of this network is a $C_4$-free loop-less symmetric digraph with only negative signs. In this way, the corollary above suggests that the fact that all the interactions are of inhibition type allows to have the greatest possible number of fixed points.

Corollary \ref{cor} contrasts with the following complexity result, which is our second main result.

\begin{theorem}\label{thm:NP-hard}
Given a graph $G$, it is coNP-hard to decide if equality \eqref{eq:fix=mis} holds,  even if $G$ has a unique induced copy of $C_4$. 
\end{theorem}

\paragraph{Graph theoretical formulation} 

Our results can be restated in purely graph theoretical terms. Given a graph $G$, let $\H(G)$ be the set of graphs obtained from $G$ by contracting some edges. For every connected subgraph $C$ of $G$, let $G/C$ be the graph obtained by contracting $C$ into a single vertex $c$ and by adding an edge between $c$ and a new vertex $c'$. Let $\H'(G)$ be the set of graphs that can be obtained from $G$ by repeating such an operation. By convention, $G$ is a member of $\H(G)$ and $\H'(G)$. We will prove the following equality.

\begin{theorem}\label{thm:H'} 
For every graph $G$, 
\[
\max_\sigma \fix(G_\sigma)=\max_{H\in \H'(G)}\mis(H).
\]
\end{theorem}

Since every graph $H$ in $\H(G)$ can be obtained from a graph $H'$ in $\H'(G)$ by removing some pending vertices, and since $H$ is then an induced subgraph of $H'$, we have $\mis(H)\leq\mis(H')$, and as a consequence $\max_{H\in \H(G)}\mis(H) \leq \max_{H\in\H'(G)}\mis(H)$. Putting things together, we then obtain the following graph theoretical version of our main result. 

\begin{corollary}
For every graph $G$,
\[
\mis(G) \leq \max_{H\in \H(G)}\mis(H) \leq \max_{H\in\H'(G)}\mis(H)\leq \left(\frac{3}{2}\right)^{m(G)}\mis(G).
\]
\end{corollary}

The upper bound on $\mis(\H'(G))$ is reached for the disjoint union of~$C_4$. Suppose indeed that $G$ is the disjoint union of $k$ copies of $C_4$. Then $m(G)=k$ and $\mis(G)=\mis(C_4)^k=2^k$ thus the upper bound is $3^k$. Now, let $H$ be the disjoint union of $k$ copies of $C_3$. Then $H\in \H(G)$ and $\mis(H)=\mis(C_3)^k=3^k$. Thus $\max_{H\in \H(G)}\mis(H)$ reaches the bound, and this forces $\max_{H\in \H'(G)}\mis(H)$ to reach the bound. Thus, as announced above, for every $k\geq 1$ there exists a graph $G$ such that $m(G)=k$ and $\max_\sigma \fix(G_\sigma)=(\frac{3}{2})^k\mis(G)$. 

Obviously, if $G$ has no induced $C_4$ then $m(G)=0$ and we obtain the following corollary. 

\begin{corollary}\label{cor:graph}
For every graph $G$ without an induced copy of $C_4$,  
\begin{equation}\label{eq:mis=mis}
\max_{H\in \H(G)}\mis(H)=\mis(G).
\end{equation}
\end{corollary}

This again contrasts with the following complexity result (which is not equivalent to Theorem~\ref{thm:NP-hard}, since  $\max_{H\in \H(G)}\mis(H)$ may be fewer than $\max_{H\in \H'(G)}\mis(H)$).

\begin{theorem}\label{thm:NP-hard2}
Given an graph $G$, it is coNP-hard to decide if equality \eqref{eq:mis=mis} holds, even if $G$ has a unique induced copy of $C_4$. 
\end{theorem}

The above graph theoretical formulations need much less material to be stated. They seem, however, more difficult to interpret than the network formulations. Beside, from a technical point of view, the induced $C_4$, which play a key role, seem more easy to manipulate with the network versions. Take for example Corollary~\ref{cor:graph}. It states that the contraction of any set of edges in a $C_4$-free graph cannot increase the number of maximal independent sets. However, the contraction of a single edge could produce many induced $C_4$, and this prevents a proof by induction on the number of contractions. This is for this kind of phenomena that the network setting seems more convenient.  

\paragraph{Organization} 
The paper is organized as follows. In Section~\ref{sec:preliminaries} we define with more precision the notions involved in the mentioned results. Then we prove Theorems~\ref{thm:main} and \ref{thm:H'} in Sections~\ref{sec:main}, \ref{sec:H'}, respectively. Theorems \ref{thm:NP-hard} and \ref{thm:NP-hard2} are then proved simultaneously in Section \ref{sec:NP-Hard}.

\section{Preliminaries}\label{sec:preliminaries}

The vertex set of a digraph $G$ is denoted $V=V(G)$ and its arc set is denoted $E=E(G)$. An arc from a vertex $u$ to $v$ is denoted $uv$. An arc $vv$ is a {\em loop}. The in-neighbor of a vertex $v$ is denoted $N_G(v)$, and if $U$ is a set of vertices then $N_G(U)=\bigcup_{v\in U}N_G(v)$. We say that $G$ is {\em symmetric} if $uv\in E$ for every $vu\in E$. The strongly connected components of $G$ are viewed as set of vertices, and a component $C$ is {\em trivial} if it contains a unique vertex. We viewed {\em (undirected) graphs} as loop-less symmetric digraphs. The set of maximal independent sets of $G$ is denoted $\Mis(G)$ and $\mis(G)=|\Mis(G)|$. The subgraph of $G$ induced by a set of vertices $U$ is denoted $G[U]$, and $G-U$ denotes the subgraph induced by $V-U$. 

A {\em signed digraph} $G_\s$ is a digraph $G$ together with an arc-labeling function $\s$ that gives a positive or negative sign to each arc of $G$ (zero sign is not considered here). Such a labeling is called a {\em repartition of signs} in $G$. 
If $U\subseteq V$, then $G_\s[U]=G[U]_{\s_{|U}}$ and $G_\s-U=G_\s[V-U]$. We say that $G_\s$ is a {\em simple signed graph} if $G$ is a graph and $\s$ is {\em symmetric}, that is, $\s(uv)=\s(vu)$ for all $uv\in E$. Equivalently, $G_\s$ is a simple signed graph if $G$, $G^+$ and $G^-$ are graphs (see an illustration in Figure~\ref{fig:signed_graph_9}). In a simple signed graph, directions do not matter, and we thus speak about positive and negative {\em edges}, instead of positive and negative arcs. 

\begin{figure}[hp]
\centering
\newcommand{\nodosSG}{
\node[vnode] (1) at (0,0){};
\node[vnode,shift={(45:\unit)}] (2) at (1){};
\node[vnode,shift={(-45:\unit)}] (3) at (1){};
\node[vnode,shift={(45:\unit)}] (4) at (2){};
\node[vnode,shift={(-45:\unit)}] (5) at (2){};
\node[vnode,shift={(-45:\unit)}] (6) at (3){};
\node[vnode,shift={(45:\unit)}] (7) at (5){};
\node[vnode,shift={(65:1.1\unit)}] (8) at (6){};
\node[vnode,shift={(-45:\unit)}] (9) at (7){};
\draw (5) circle (2\unit);
}
\begin{tabular}{ccc}
\begin{tikzpicture}
 \nodosSG
 \path[bend right=17]
 (1) edge[negative] (2)
 (2) edge[positive] (1)
 (1) edge[negative] (3)
 (3) edge[negative] (1)
 (2) edge[positive] (4)
 (4) edge[negative] (2)
 (2) edge[negative] (5)
 (5) edge[positive] (2)
 (3) edge[negative] (5)
 (5) edge[positive] (3)
 (3) edge[positive] (6)
 (6) edge[negative] (3)
 (4) edge[positive] (7)
 (7) edge[negative] (4)
 (5) edge[positive] (7)
 (7) edge[positive] (5)
 (6) edge[positive] (8)
 (8) edge[negative] (6)
 (6) edge[negative, bend right=80] (9)
 (9) edge[positive, bend left=40] (6)
 (9) edge[negative] (7)
 (7) edge[positive] (9)
 (9) edge[positive] (8)
 (8) edge[negative] (9)
;
\end{tikzpicture}& \begin{tikzpicture}
 \nodosSG
 \path[bend right=17,positive]
 (2) edge (1)
 (2) edge (4)
 (5) edge (2)
 (5) edge (3)
 (3) edge (6)
 (4) edge (7)
 (5) edge (7)
 (7) edge (5)
 (6) edge (8)
 (9) edge[bend left=40] (6)
 (7) edge (9)
 (9) edge (8)
;
\end{tikzpicture}&\begin{tikzpicture}
 \nodosSG
 \path[bend right=17,negative]
 (1) edge (2)
 (1) edge (3)
 (3) edge (1)
 (4) edge (2)
 (2) edge (5)
 (3) edge (5)
 (6) edge (3)
 (7) edge (4)
 (8) edge (6)
 (6) edge[bend right=80] (9)
 (9) edge (7)
 (8) edge (9)
;
\end{tikzpicture}\\
$G_{\s_1}$ & $G^+_{\s_1}$& $G^-_{\s_1}$\\[4mm]
\begin{tikzpicture}
 \nodosSG
 \path[bend right=17]
 (1) edge[negative] (2)
 (2) edge[negative] (1)
 (1) edge[negative] (3)
 (3) edge[negative] (1)
 (2) edge[positive] (4)
 (4) edge[positive] (2)
 (2) edge[positive] (5)
 (5) edge[positive] (2)
 (3) edge[negative] (5)
 (5) edge[negative] (3)
 (3) edge[positive] (6)
 (6) edge[positive] (3)
 (4) edge[positive] (7)
 (7) edge[positive] (4)
 (5) edge[positive] (7)
 (7) edge[positive] (5)
 (6) edge[negative] (8)
 (8) edge[negative] (6)
 (6) edge[negative, bend right=80] (9)
 (9) edge[negative, bend left=40] (6)
 (9) edge[negative] (7)
 (7) edge[negative] (9)
 (9) edge[negative] (8)
 (8) edge[negative] (9)
;
\end{tikzpicture}&\begin{tikzpicture}
 \nodosSG
 \path[bend right=17,positive]
 (2) edge (4)
 (4) edge (2)
 (2) edge (5)
 (5) edge (2)
 (3) edge (6)
 (6) edge (3)
 (4) edge (7)
 (7) edge (4)
 (5) edge (7)
 (7) edge (5)
;
\end{tikzpicture}& \begin{tikzpicture}
 \nodosSG
 \path[bend right=17,negative]
 (1) edge (2)
 (2) edge (1)
 (1) edge (3)
 (3) edge (1)
 (3) edge (5)
 (5) edge (3)
 (6) edge (8)
 (8) edge (6)
 (6) edge[bend right=80] (9)
 (9) edge[bend left=40] (6)
 (9) edge (7)
 (7) edge (9)
 (9) edge (8)
 (8) edge (9)
;
\end{tikzpicture}\\
$G_{\s_2}$ & $G^+_{\s_2}$& $G^-_{\s_2}$\\[4mm]
\begin{tikzpicture}
 \nodosSG
 \path
 (1) edge[dneg] (2)
 (1) edge[dneg] (3)
 (2) edge[dpos] (4)
 (2) edge[dpos] (5)
 (3) edge[dneg] (5)
 (3) edge[dpos] (6)
 (4) edge[dpos] (7)
 (5) edge[dpos] (7)
 (6) edge[dneg] (8)
 (6) edge[dneg] (9)
 (9) edge[dneg] (7)
 (9) edge[dneg] (8)
;
\end{tikzpicture}&\begin{tikzpicture}
 \nodosSG
 \path[dpos]
 (2) edge (4)
 (2) edge (5)
 (3) edge (6)
 (4) edge (7)
 (5) edge (7)
;
\end{tikzpicture}& \begin{tikzpicture}
 \nodosSG
 \path[dneg]
 (1) edge (2)
 (1) edge (3)
 (3) edge (5)
 (6) edge (8)
 (6) edge (9)
 (9) edge (7)
 (9) edge (8)
;
\end{tikzpicture}\\ 
Other representation of $G_{\s_2}$ & Other representation of $G^+_{\s_2}$& Other representation of $G^-_{\s_2}$
\end{tabular}
\caption{A signed graph $G_{\s_1}$ and a {\em simple} signed graph $G_{\s_2}$.}\label{fig:signed_graph_9}
\end{figure}

A {\em Boolean network} on a finite set $V$ is a function 
\[
f:\B^V\to\B^V,\qquad x=(x_v)_{v\in V}\mapsto f(x)=(f_v(x))_{v\in V}
\]
Thus each component $f_v$ is a Boolean function from $\B^V$ to $\B$ (often called local update function). The {\em interaction graph} of $f$ is the digraph $G$ on $V$ such that for all $u,v\in V$ there is an $uv$ if and only if $f_v$ depends on $x_u$, that is, there exists $x,y\in\B^V$ that only differs in $x_u\neq y_u$ such that $f_v(x)\neq f_v(y)$. We say that $f$ is a {\em conjunctive network} if, for all $v\in V$, $f_v$ is a conjunction of positive and negative literals. If $f$ is a conjunctive network, then the {\em signed interaction graph} of $f$ is the signed digraph $G_\s$ where $G$ is the interaction graph of $f$, and where for all arc $uv$ we have $\s(uv)=+$ if $x_u$ is a positive literal of $f_v$ and $\s(uv)=-$ if $\overline{x_u}$ is a negative literal of $f_v$. Conversely, given any signed digraph $G_\s$, there is clearly a unique conjunctive network whose signed interaction graph is $G_\s$, namely the conjunctive network $f^{G_\s}$ defined by 
\[
\forall v\in V,\quad\forall x\in\B^V, 
\qquad f^{G_\s}_v(x)=\prod_{u\in N_{G^+_\s}(v)} x_u\prod_{u\in N_{G^-_\s}(v)}1-x_u.
\]
We call $f^{G_\s}$ the conjunctive network {\em on} $G_\s$. We denote by $\Fix(G_\s)$ the set of fixed points of $f^{G_\s}$ and by $\fix(G_\s)$ the number of fixed points in $f^{G_\s}$. 
%
%

We will often see a fixed point as the characteristic function of a subset of $V$. For that we set $\ONE(x)=\{v\in V\,|\,x_v=1\}$, and we say that a subset $S\subseteq V$ is a {\em fixed set} of $G_\s$ if $S=\ONE(x)$ for some fixed point $x\in\Fix(G_\s)$. By abuse of notation, the set of fixed sets of $G_\s$ is also denoted $\Fix(G_\s)$. Thus, for all $S\subseteq V$, we have $S\in\Fix(G_\s)$ if and only if for all $v\in V$ we have  
\begin{equation}\label{eq:fixedset}
v\in S~\iff~ N_{G_\s^-}(v)\cap S=\emptyset\text{ and }N_{G_\s^+}(v)\subseteq S.
\end{equation}

Let $G_-$ be the signed digraph obtain from $G$ by labeling negatively every arc. Then $G^-_-=G$ and $G^+_-=\emptyset$, thus the above equivalence becomes: $v\in S$ if and only if $N_G(v)\cap S=\emptyset$. If $G$ is a graph, this is precisely the definition of a maximal independent set. Thus we have the following basic relation between fixed sets and maximal independent sets, already exhibited in \cite{ARS14}.

\begin{proposition}\label{pro:mis}
For every graph $G$ we have $\Fix(G_-)=\Mis(G)$.
\end{proposition}

We immediately obtain
\begin{equation}
\mis(G)\leq \max_\sigma\fix(G_\sigma),
\end{equation}
and thus only the upper bound in Theorem~\ref{thm:main} has to be proven. 

Another important observation is that for every signed digraph $G_\s$, if $C$ is a strongly connected component of $G^+_\s$, then for all $S\in\Fix(G_\s)$ we have either $C\cap S=\emptyset$ or $C\subseteq S$. As a consequence, if $G^+_\s$ is strongly connected and $S\in\Fix(G_\s)$ then either $S=\emptyset$ or $S=V$. Thus $G_\s$ has at most two fixed sets: $\emptyset$ and $V$. Clearly, $V$ is a fixed set if and only if $G_\s$ has no negative arcs, and $\emptyset$ is a fixed set in any case. We have thus the following proposition. 

\begin{proposition}\label{pro:basic1}
For every signed digraph $G_\s$ such that $G^+_\s$ is strongly connected, we have $\Fix(G_\s)=\{\emptyset\}$ if $G_\sigma$ has at least one negative arc, and $\Fix(G_\s)=\{\emptyset,V\}$ otherwise. 
\end{proposition}

\noindent
{\em Remark about notations.} In the following, we mainly consider signed digraphs. In order to simplify notations, given a signed digraph $G=H_\s$, we set $G^+=H^+_\s$ and $G^-=H^-_\s$. Hence, if we consider a signed digraph $G$, the underlying repartition of sign is given by the two spanning subgraphs $G^+$ and $G^-$. Furthermore, all concepts that do not involve sign are applied on $G$ or its underlying unsigned digraph $H$ indifferently. For instance, we write $\mis(G)$ or $\mis(H)$ indifferently.    

\section{Proof of Theorem~\ref{thm:main}}\label{sec:main}

As said above, we only have to prove the upper-bound. For a graph $G$, let $m'(G)$ be the maximum size of a matching $M$ of $G$ such that every edge $uv$ of $M$ is contained in an induced copy of $C_4$ in which $u$ and $v$ are the only vertices adjacent to an edge of $M$. Clearly, we have $m'(G)\leq m(G)$, see Figure \ref{fig_m} for an illustration. The upper bound that we will prove, the following, is thus slightly stronger than that of Theorem~\ref{thm:main}.  

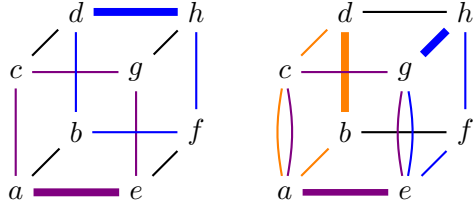
\begin{figure}
\[
\begin{array}{ccc}
\begin{tikzpicture}
\node (000) at (-1.2,-1.2){$a$};
\node (001) at (-0.4,-0.4){$b$};
\node (010) at (-1.2, 0.4){$c$};
\node (011) at (-0.4, 1.2){$d$};
\node (100) at ( 0.4,-1.2){$e$};
\node (101) at ( 1.2,-0.4){$f$};
\node (110) at ( 0.4, 0.4){$g$};
\node (111) at ( 1.2, 1.2){$h$};
\path[thick]
(000) edge[violet,line width=3] (100)
(000) edge[violet] (010)
(000) edge (001)
(100) edge[violet] (110)
(100) edge (101)
(010) edge[violet] (110)
(010) edge (011)
(001) edge[blue] (101) 
(001) edge[blue] (011)
(110) edge (111)
(101) edge[blue] (111)
(011) edge[blue,line width=3] (111)
;
\end{tikzpicture}
&&
\begin{tikzpicture}
\node (000) at (-1.2,-1.2){$a$};
\node (001) at (-0.4,-0.4){$b$};
\node (010) at (-1.2, 0.4){$c$};
\node (011) at (-0.4, 1.2){$d$};
\node (100) at ( 0.4,-1.2){$e$};
\node (101) at ( 1.2,-0.4){$f$};
\node (110) at ( 0.4, 0.4){$g$};
\node (111) at ( 1.2, 1.2){$h$};
\path[thick]
(011) edge (111)
(101) edge[blue] (111)
(110) edge[blue,line width=3] (111)
(001) edge[orange,line width=3] (011)
(001) edge (101) 
(010) edge[orange] (011)
(010) edge[violet] (110)
(100) edge[blue] (101)
(100) edge[blue,bend right=10] (110)
(100) edge[violet,bend left=10] (110)
(000) edge[orange] (001)
(000) edge[orange,bend left=10] (010)
(000) edge[violet,bend right=10] (010)
(000) edge[violet,line width=2.5] (100)
;
\end{tikzpicture}
\end{array}
\]
\caption{For the three-dimensional cube $Q_3$, we have $m'(Q_3)=2$ and $m(Q_3)=3$. In the left, the matching $M=\{ae,dh\}$ illustrates the fact that $m'(Q_3)=2$: the edge $ae$ belongs to an induced $C_4$ which does not contain $d$ neither $h$, namely $aegc$; and the edge $dh$ belongs to an induced $C_4$ which does not contain $a$ neither $e$, namely $dhfb$. In the right, the matching $M=\{ae,gh,db\}$ illustrates the fact that $m(Q_3)=3$: the edge $ae$ belongs to an induced $C_4$ which does not contain $gh$ neither $db$, namely $aegc$; the edge $gh$ belongs to an induced $C_4$ which does not contain $ae$ neither $db$, namely $ghfe$; and the edge $db$ belongs to an induced $C_4$ which does not contain $ae$ neither $gh$, namely $dbac$.}\label{fig_m}
\end{figure}

\begin{theorem}\label{thm:upper_bound}
For every signed graph $G$,
\[
\fix(G)\leq\left(\frac{3}{2}\right)^{m'(G)}\mis(G).
\]
\end{theorem}

The proof is based on the following three lemmas, which are proved in Sections \ref{sec:symmetrization}, \ref{sec:reduction} and \ref{sec:suppression} respectively.

\begin{lemma}\label{lemma:symmetric}
For every signed graph $G$ there exists a simple signed graph $G'$ obtained from $G$ by changing the repartition of signs such that $\fix(G)\leq\fix(G')$.
\end{lemma}

\begin{lemma}\label{lemma:matching}
For every simple signed graph $G$ there exists a simple signed graph $G'$ obtained from $G$ by changing the repartition of sign such that $\fix(G)\leq\fix(G')$ and such that the set of positive edges of $G'$ is a matching.
\end{lemma}

\begin{lemma}\label{lemma:C_4}
Let $G$ be a simple signed graph whose set of positive edges is a matching. Let $G'$ be the simple signed graph obtained from $G$ by making negative a positive edge $ab$. If $G$ has an induced copy of $C_4$ containing $a$ and $b$ and no other vertices adjacent to a positive edge, then $\fix(G)\leq (3/2)\,\fix(G')$ and otherwise $\fix(G)\leq \fix(G')$. 
\end{lemma}

\begin{proof}[Proof of Theorem \ref{thm:upper_bound}, assuming Lemmas~\ref{lemma:symmetric}, \ref{lemma:matching} and \ref{lemma:C_4}]
Let $G$ be a signed graph. The following three previous lemmas show that we can change the repartition of sign in $G$ in order to obtain a simple signed graph $G_0$ with the following properties: $\fix(G)\leq\fix(G_0)$; the set of positive edges of $G_0$ forms a matching, say $\{e_1,\dots,e_m\}$; and each positive edge $e_k$ belongs to an induced copy of $C_4$ in which the only vertices adjacent to a positive edge are the two vertices of $e_k$. Thus, by the definition of $m'(G)$, we have 
\begin{equation}\label{eq:mainproof1}
m\leq m'(G).
\end{equation}
Now, for $0\leq k<m$, let $G_{k+1}$ be the simple signed graph obtained from $G_k$ by making negative the positive edge $e_{k+1}$. By Lemma~\ref{lemma:C_4}, we have $\fix(G_k)\leq(3/2)\,\fix(G_{k+1})$. Thus 
\begin{equation}\label{eq:mainproof2}
\fix(G)\leq \fix(G_0)\leq(3/2)^m\,\fix(G_m).
\end{equation}
Furthermore, since $G_m$ has only negative edges, by Proposition~\ref{pro:mis}, we have $\fix(G_m)=\mis(G_m)=\mis(G)$, and with \eqref{eq:mainproof1} and \eqref{eq:mainproof2} we obtain the theorem. 
\end{proof}

\subsection{Proof of Lemma~\ref{lemma:symmetric} (symmetrization of signs)}\label{sec:symmetrization}

\begin{lemma}\label{lemma:cambiaArco}
Let $G$ be a signed graph, and let $v$ be a vertex of $G$. Let $A(v)$ be the set of arcs $uv$ such that $uv$ and $vu$ have different signs, and suppose that $A(v)$ contains at least one positive arc. Let $G'$ be the signed graph obtained from $G$ by changing the sign of each arc in $A(v)$. Then $\fix(G)\leq \fix(G')$.
\end{lemma}

\begin{figure}[!htb]
\centering
\def\invo{\textcolor{white}{o}}
\begin{tikzpicture}
\draw (-4\unit,0) circle (2.4\unit);
\node[texto,above] (l) at (-4\unit,2.6\unit) {$G$};
\node[inode,texto] (v) at (-4\unit,0){$v$};
\node[inode,shift={(45:2\unit)},texto] (1) at (v) {\invo};
\node[inode,shift={(135:2\unit)},texto] (2) at (v) {\invo};
\node[inode,shift={(225:2\unit)},texto] (3) at (v) {\invo};
\node[inode,shift={(-45:2\unit)},texto] (4) at (v) {\invo};
\path[bend right=15]
(v) edge[positive] (1)
(1) edge[positive] (v)
(v) edge[negative] (2)
(2) edge[negative] (v)
(v) edge[positive] (3)
(3) edge[negative] (v)
(v) edge[negative] (4)
(4) edge[positive] (v)
;
\node[texto,shift={(270:3.4\unit)}] (a) at (v) {$A(v)$};
\draw[-to,dotted] (a.110) to +(100:2.2\unit);
\draw[-to,dotted] (a.70) to  +(78:2.6\unit);
\flecha{(-\unit,0)}{2\unit}
\draw (4\unit,0) circle (2.4\unit);
\node[texto,above] (l) at (4\unit,2.6\unit) {$G'$};
\node[inode,texto] (v) at (4\unit,0){$v$};
\node[inode,shift={(45:2\unit)},texto] (1) at (v) {\invo};
\node[inode,shift={(135:2\unit)},texto] (2) at (v) {\invo};
\node[inode,shift={(225:2\unit)},texto] (3) at (v) {\invo};
\node[inode,shift={(-45:2\unit)},texto] (4) at (v) {\invo};
\path[bend right=15]
(v) edge[positive] (1)
(1) edge[positive] (v)
(v) edge[negative] (2)
(2) edge[negative] (v)
(v) edge[positive] (3)
(3) edge[positive] (v)
(v) edge[negative] (4)
(4) edge[negative] (v)
;
\end{tikzpicture}
\caption{Illustration of the transformation described in Lemma~\ref{lemma:cambiaArco}.}\label{fig:symmetric}
\end{figure}
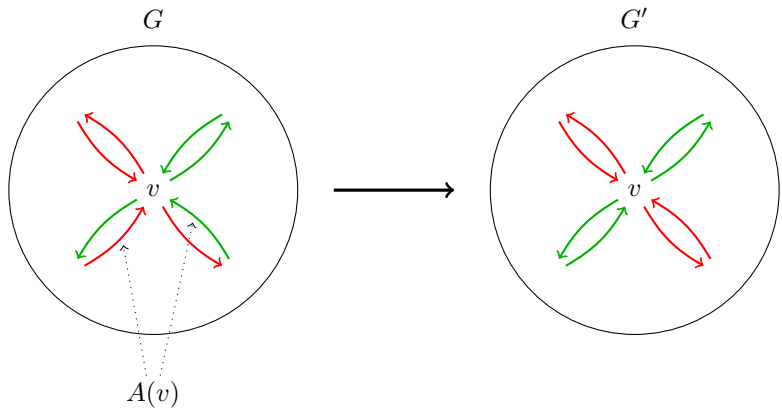

\begin{proof}
Let $f$ and $f'$ be the conjunctive networks on $G$ and $G'$ respectively. For each $x\in\Fix(G)$, let $x'$ be defined by 
\[
x'_v=f'_v(x)\quad\text{and}\quad x'_u=x_u~\text{for all } u\neq v.
\]
We will prove that $x\mapsto x'$ is an injection from $\Fix(G)$ to $\Fix(G')$. To prove the injectivity, it is sufficient to prove that 
\[
\forall x\in\Fix(G),\qquad x_v=0. 
\]
Let $x\in\Fix(G)$ and let $uv$ be a positive arc of $A(v)$. If $x_v=1=f_v(x)$ then $x_u=1$ and since $vu$ is negative, we have $f_u(x)=0\neq x_u$, a contradiction. 

Now, let us prove that $x'\in\Fix(G')$. Suppose, by contradiction, that $f'(x')\neq x'$. Since $G'$ has no loop on $v$ we have $x'_v=f'_v(x)=f'_v(x')$ thus there exists $u\neq v$ such that 
\[
f_u(x')=f'_u(x')\neq x'_u=x_u=f_u(x).
\]
Thus $x\neq x'$. It means that $x_v=0$ and $x'_v=1=f'_v(x)$, and that $vu$ is an arc of $G$. Suppose first that $vu$ is positive. Since $x_v=0$ we have $f_u(x)=0$ thus $x_u=0$, and since $f'_v(x)=1$ we deduce that $uv$ is a negative arc of $G'$. So $uv$ and $vu$ have different signs in $G'$, a contradiction. Suppose now that $vu$ is negative. Since $x'_v=1$ we have $f_u(x')=0$ thus $x_u=1$, and since $f'_v(x)=1$ we deduce that $uv$ is a positive arc of $G'$. Thus $uv$ and $vu$ have different signs in $G'$, a contradiction. So $x'\in\Fix(G')$.  
\end{proof}

\begin{proof}[Proof of Lemma~\ref{lemma:symmetric}]
Suppose that $G$ is a signed graph which is not simple. Then it contains at least one asymmetry of signs, that is, an arc $uv$ such that $uv$ and $vu$ have not the same sign. If $uv$ is positive then, by changing the sign of the arcs in $A(v)$, we obtained a signed graph $G'$ with strictly less asymmetries of signs and, by the previous lemma, $\fix(G)\leq\fix(G')$. Otherwise, $vu$ is positive and thus, by changing the sign of the arcs in $A(u)$, we also strictly decrease the number of asymmetries of signs without decreasing the number of fixed points. We then obtain the lemma by applying this transformation until that there is no asymmetry of signs.
\end{proof}

\subsection{Proof of Lemma~\ref{lemma:matching} (reduction of positive components)}\label{sec:reduction}

Let $G$ be a simple signed graph with vertex set $V$. In this section, we prove that we can change the repartition of signs, without decreasing the number of fixed sets, until to obtain a simple signed graph $G'$ in which positive edges form a matching. 

For that we will introduced decomposition properties, used throughout the paper. Actually, if $U\subseteq V$ is such that $G$ has no positive edge between $U$ and $V-U$, then a lot of things on $\Fix(G)$ can be understood from $\fix(G[U])$ and the fixed sets of the induced subgraphs of $G-U$. In particular, under the condition that there is no positive edge between $U$ and $V-U$, we will prove that if $S\in\Fix(G[U])$ and $S'\in\Fix(G[U'])$, where $U'$ is some subset of $V-U$ that only depends on $U$ and $S$, then $S\cup S'\in\Fix(G)$. This very useful decomposition property, essential for the rest of the proof, is based on the following definitions.   

For all $U\subseteq V$, we denote by $\N_G(U)$ the union of $N_G(U)$ and the connected components $C$ of $G^+$ such that $N_G(U)\cap C\neq\emptyset$. In other words, $\N_G(U)$ is the union of $N_G(U)$ and the set of vertices reachable from $N_G(U)$ with a path that contains only positive edges. Note that for all $U\subseteq V$, $G$ has no positive edges between $\N_G(U)$ and $V-\N_G(U)$. We can see an illustration of $N_G(U)$ and $\N_G(U)$ in Figure~\ref{fig:NGU}. 

\begin{figure}[!htb]
 \centering
\begin{tikzpicture}[scale=0.6]
\foreach \x/\y in {1/5,2/275,3/185,4/95}
 \coordinate (\x) at (\y: 1.5 and 1);
\foreach \x/\y in {5/60,6/30,7/0,8/330,9/300,10/270,11/240,12/210,13/180,14/150,15/120,16/90}
 \coordinate (\x) at (\y: 3 and 2);
 \foreach \x/\y in {17/70,18/50,19/30,20/12,21/352,22/332,23/312,24/292,25/270,26/250,27/230,28/210,29/190,30/170,31/150,32/130,33/110,34/90}
 \coordinate (\x) at (\y: 4 and 3);
 \foreach \x/\y in {35/73,36/60,37/47,38/34,39/21,40/4,41/347,42/330,43/316,44/305,45/292,46/283,47/274,48/263,49/250,50/236,51/223,52/210,53/197,54/183,55/170,56/157,57/143,58/130,59/117,60/103,61/90}
 \coordinate (\x) at (\y: 5 and 4);
 \foreach \x/\y in{62/72,63/60,64/48,65/37,66/25,67/14,68/2,69/350,70/339,71/327,72/318,73/306,74/295,75/283,76/272,77/260,78/249,79/237,80/225,81/214,82/202,83/191,84/179,85/167,86/156,87/144,88/133,89/121,90/110,91/98,92/86}
 \coordinate (\x) at (\y: 6 and 5);
\foreach \x in{62,35,17,18,19,20,40,67,68,69,42,23,9,25,46,47,76,48,26,11,12,13,14,15,32,33,16,61,92}
  \node[circle,minimum size=3\sombra] (sg\x) at (\x) {};
  
\filldraw[fill=black!15,rounded corners,line width=0.7](sg62.0) to (sg35.0) to (sg17.45) to (sg18.50) to (sg19.30) to (sg20.45) to (sg40.135) to (sg67.90) to (sg67.0) to (sg68.30) to (sg69.0) to (sg69.225) to (sg42.290) to (sg23.330) to (sg9.270) to  (sg46.90) to (sg46.0) to (sg46.270) to (sg76.0) to (sg76.270) to (sg76.180) to (sg47.225) to (sg48.225) to (sg26.180) to (sg11.240) to (sg12.210) to (sg13.180) to (sg14.190) to (sg32.225) to (sg32.145) to (sg33.90) to (sg16.90) to (sg17.180) to (sg35.225) to (sg61.270) to (sg61.180) to (sg92.135) to (sg62.90) --cycle;

\foreach \x in {11,12,13,4,15,16,17,18,19,20,21,8,2}
  \node[circle,minimum size=2\sombra] (sm\x) at (\x) {};
\filldraw[fill=black!35,rounded corners,line width=0.7](sm17.180) to (sm17.70) to (sm18.50) to (sm19.30) to (sm20.12) to (sm21.350) to (sm8.310) to (sm2.275) to (sm11.240) to (sm12.210) to (sm13.180) to (sm13.90) to (sm4.180) to (sm15.225) to  (sm15.120) to (sm16.90)  --cycle;

\foreach \x in {2,3,4,5,6,7}
  \node[circle,minimum size=\sombra,line width=0.7] (sp\x) at (\x) {};
\filldraw[fill=black!55,rounded corners](sp2.275) to (sp3.185) to (sp4.95) to (sp5.60) to (sp6.30) to (sp7.0)--cycle;

\draw (0,0) ellipse (7 and 6);

\foreach \x in {1,2,...,92}
  \node[vnode] (n\x) at (\x){};
 \path[dpos]
 (n1) edge (n4)
 (n1) edge (n7)
 (n2) edge (n10)
 (n3) edge (n4)
 (n4) edge (n5)
 (n4) edge (n13) 
 (n5) edge (n6) 
 (n5) edge (n17) 
 (n5) edge (n18) 
 (n6) edge (n7)
 (n9) edge (n10)
 (n10) edge (n26)
 (n13) edge (n14)
 (n15) edge (n32)
 (n15) edge (n33)
 (n17) edge (n18)
 (n17) edge (n35)
 (n21) edge (n22)
 (n21) edge (n40)
 (n22) edge (n23)
 (n22) edge (n41)
 (n25) edge (n26)
 (n25) edge (n47)
 (n25) edge (n48)
 (n26) edge (n48)
 (n27) edge (n28)
 (n27) edge (n50)
 (n29) edge (n30)
 (n29) edge (n52)
 (n29) edge (n53)
 (n30) edge (n54)
 (n32) edge (n33)
 (n34) edge (n60)
 (n35) edge (n61)
 (n35) edge (n62)
 (n35) edge (n92)
 (n37) edge (n38)
 (n37) edge (n64)
 (n37) edge (n65)
 (n38) edge (n64)
 (n38) edge (n65)
 (n38) edge (n66)
 (n40) edge (n41)
 (n40) edge (n67)
 (n40) edge (n68)
 (n41) edge (n42)
 (n42) edge (n69)
 (n43) edge (n44)
 (n45) edge (n73)
 (n45) edge (n74)
 (n46) edge (n47)
 (n47) edge (n76)
 (n49) edge (n78)
 (n49) edge (n79)
 (n49) edge (n80)
 (n50) edge (n51)
 (n52) edge (n82)
 (n54) edge (n84)
 (n56) edge (n57)
 (n56) edge (n86)
 (n57) edge (n58)
 (n57) edge (n87)
 (n57) edge (n88)
 (n58) edge (n89)
 (n61) edge (n92)
 (n62) edge (n92)
 (n64) edge (n65)
 (n65) edge (n66)
 (n70) edge (n71)
 (n73) edge (n74)
 (n78) edge (n79)
 (n79) edge (n80)
 (n81) edge (n82)
 (n86) edge (n87)
 (n87) edge (n88)
 (n88) edge (n89)
 (n90) edge (n91)
 ;
 \path[dneg]
 (n1) edge (n2)
 (n1) edge (n3)
 (n1) edge (n6)
 (n2) edge (n3)
 (n2) edge (n7)
 (n2) edge (n8)
 (n2) edge (n11)
 (n3) edge (n5)
 (n3) edge (n12)
 (n3) edge (n13)
 (n4) edge (n15)
 (n4) edge (n16)
 (n5) edge (n16) 
 (n6) edge (n19) 
 (n7) edge (n8)
 (n7) edge (n19)
 (n7) edge (n20)
 (n7) edge (n21)
 (n8) edge (n9)
 (n9) edge (n24)
 (n11) edge (n27)
 (n14) edge (n30)
 (n16) edge (n33)
 (n16) edge (n34)
 (n18) edge (n36)
 (n19) edge (n37)
 (n20) edge (n39)
 (n22) edge (n24)
 (n23) edge (n24)
 (n23) edge (n42)
 (n24) edge (n25)
 (n24) edge (n43)
 (n24) edge (n44)
 (n24) edge (n45)
 (n25) edge (n46)
 (n26) edge (n49)
 (n27) edge (n49)
 (n28) edge (n29)
 (n30) edge (n31)
 (n31) edge (n32)
 (n31) edge (n56)
 (n33) edge (n59)
 (n36) edge (n37)
 (n36) edge (n63)
 (n39) edge (n66)
 (n39) edge (n67)
 (n41) edge (n69)
 (n42) edge (n43)
 (n42) edge (n70)
 (n42) edge (n71)
 (n43) edge (n71)
 (n43) edge (n72)
 (n44) edge (n45)
 (n44) edge (n72)
 (n45) edge (n46)
 (n46) edge (n75)
 (n47) edge (n77)
 (n48) edge (n49)
 (n48) edge (n77)
 (n50) edge (n81)
 (n51) edge (n52)
 (n51) edge (n81)
 (n52) edge (n83)
 (n53) edge (n54)
 (n53) edge (n83)
 (n55) edge (n56)
 (n55) edge (n85)
 (n57) edge (n86)
 (n58) edge (n59)
 (n58) edge (n88)
 (n60) edge (n61)
 (n60) edge (n90)
 (n61) edge (n91)
 (n61) edge (n62)
 (n66) edge (n67)
 (n68) edge (n69)
 (n69) edge (n70)
 (n71) edge (n72)
 (n72) edge (n73)
 (n74) edge (n75)
 (n76) edge (n77)
 (n80) edge (n81)
 (n82) edge (n83)
 (n85) edge (n86)
 (n89) edge (n90)
 (n91) edge (n92)
;

\node[texto,inner sep=2pt](l) at (20:8.5) {$\N_G(U)$};
\draw[-to,line width=1] (l.210) -- +(-1.3,-1.7);
\node[texto,inner sep=2pt](l) at (37:8.2) {$N_G(U)$};
\draw[-to,line width=1] (l.230) -- +(-2.5,-3.5);
\node[texto,inner sep=2pt](l) at (65:7.7) {$U$};
\draw[-to,line width=1] (l.240) -- +(-1.5,-5.9);

\end{tikzpicture}
\caption{Illustration of $\N_G(U)$.}\label{fig:NGU}
\end{figure}
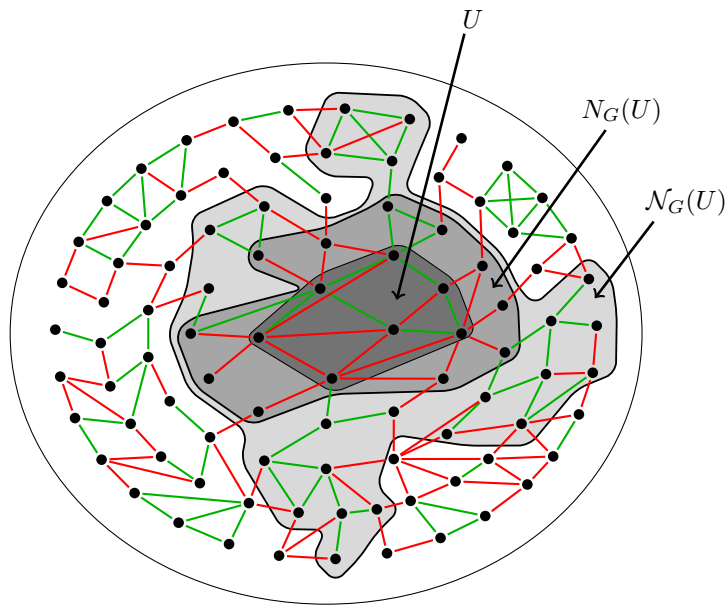

We are now in position to introduce our decomposition tool. For all $U\subseteq V$ we set
\[
\Fix(G,U)=\{S\cup S'\,|\,S\in\Fix(G[U]),~S'\in\Fix(G-U-\N_G(S))\}
\]
and $\fix(G,U)=|\Fix(G,U)|$. 

\begin{example}
Here is an example where $G[U]$ contains 2 fixed sets, $S_1$ and $S_2$: 

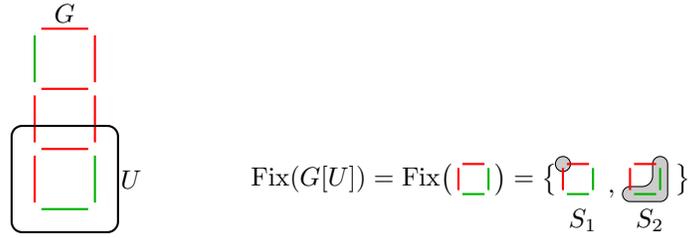
\begin{center}
\begin{tikzpicture}
 \node[inode] (1) at (-0.5\unit,-0.5\unit) {};
 \node[inode] (2) at (0.5\unit,-0.5\unit) {};
 \node[inode] (3) at (-0.5\unit,0.5\unit) {};
 \node[inode] (4) at (0.5\unit,0.5\unit) {};
 \node[inode] (5) at (-0.5\unit,1.5\unit) {};
 \node[inode] (6) at (0.5\unit,1.5\unit) {};
 \node[inode] (7) at (-0.5\unit,2.5\unit) {};
 \node[inode] (8) at (0.5\unit,2.5\unit) {};
 \path[dneg]
 (1) edge[dpos] (2)
 (1) edge (3)
 (2) edge[dpos] (4)
 (3) edge (4)
 (3) edge (5)
 (4) edge (6)
 (5) edge (6)
 (5) edge[dpos] (7)
 (6) edge (8)
 (7) edge (8)
 ;
 \draw[cuadros] (225:1.25\unit) rectangle +(45:2.5\unit);
 \node[texto] (U) at (1.1\unit,0) {$U$};
 \node[texto] (G) at (0,2.75\unit) {$G$};
 \node[right,texto] (l) at (3\unit,0) {$\Fix(G[U])=\Fix\big($};

 \node[inode,minimum size=1pt,shift={(0,-0.5\punit)}] (1) at (l.0) {};
 \node[inode,minimum size=1pt,shift={(\punit,-0.5\punit)}] (2) at (l.0)  {};
 \node[inode,minimum size=1pt,shift={(0,0.5\punit)}] (3) at (l.0)  {};
 \node[inode,minimum size=1pt,shift={(\punit,0.5\punit)}] (4) at (l.0)  {};
 \path
 (1) edge[dpos] (2)
 (3) edge[dneg] (4)
 (1) edge[dneg] (3)
 (2) edge[dpos] (4)
 ;
 \node[texto,right,shift={(0,0.5\punit)}] (l) at (2) {$\big)=\big\{$};
 \coordinate[shift={(0,-0.5\punit)}] (c1) at (l.0)  {};
 \coordinate[shift={(\punit,-0.5\punit)}] (c2)  at (l.0) {};
 \coordinate[shift={(0,0.5\punit)}] (c3)  at (l.0) {};
 \coordinate[shift={(\punit,0.5\punit)}] (c4)  at (l.0) {};
 \node[inode,minimum size=1pt] (1) at (c1) {};
 \node[inode,minimum size=1pt] (2) at (c2) {};
 \node[inode,minimum size=1pt] (3) at (c3) {};
 \node[inode,minimum size=1pt] (4) at (c4) {};
 \filldraw[fill=black!20] (c3) circle (\ssombra);
 \path
 (1) edge[dpos] (2)
 (3) edge[dneg] (4)
 (1) edge[dneg] (3)
 (2) edge[dpos] (4)
 ;
 \node[below right,texto,yshift=-\sep] (ll) at (1) {$S_1$};
 \node[texto,shift={(\sep,0)},right] (l3) at (2) {$,$};
\coordinate[shift={(\sep,0)}] (c1) at (l3.0) {};
 \coordinate[shift={(\punit,0)}] (c2) at (c1) {};
 \coordinate[shift={(0,\punit)}] (c3) at (c1) {};
 \coordinate[shift={(\punit,\punit)}] (c4) at (c1) {};
 \node[inode,minimum size=1pt] (1) at (c1) {};
 \node[inode,minimum size=1pt] (2) at (c2) {};
 \node[inode,minimum size=1pt] (3) at (c3) {};
 \node[inode,minimum size=1pt] (4) at (c4) {};
   \filldraw[fill=black!20] (c2)++(-\ssombra,\ssombra)--++(-\punit+\ssombra,0) arc (90:270:\ssombra)--++(\punit,0) arc (270:360:\ssombra)-- ++(0,\punit)  arc (0:180:\ssombra){[rounded corners]--cycle};
 \path
 (1) edge[dpos] (2)
 (3) edge[dneg] (4)
 (1) edge[dneg] (3)
 (2) edge[dpos] (4)
 ;
 \node[below right,texto,yshift=-\sep] (ll) at (1) {$S_2$};
 \node[texto,right,shift={(\sep,0.5\punit)}] (l) at (c2) {$\big\}$};
\end{tikzpicture}
\end{center}

The induced subgraph $G-U-\N_G(S_1)$ has two fixed sets, $S_{11}$ and $S_{12}$: 

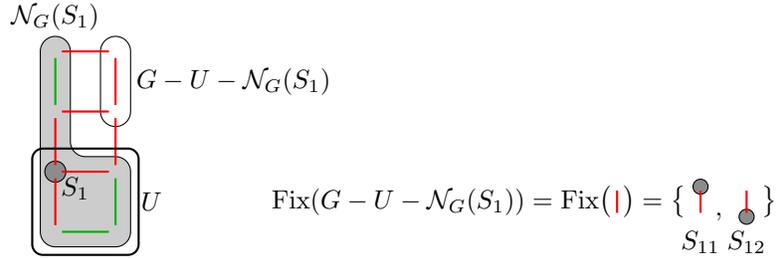
\begin{center}
\begin{tikzpicture}
 \coordinate (c1) at (-0.5\unit,-0.5\unit) ;
 \coordinate (c2) at (0.5\unit,-0.5\unit) ;
 \coordinate (c3) at (-0.5\unit,0.5\unit) ;
 \coordinate (c4) at (0.5\unit,0.5\unit) ;
 \coordinate (c5) at (-0.5\unit,1.5\unit) ;
 \coordinate (c6) at (0.5\unit,1.5\unit) ;
 \coordinate (c7) at (-0.5\unit,2.5\unit) ;
 \coordinate (c8) at (0.5\unit,2.5\unit) ;
 \filldraw[fill=black!20] (c7)++(0:\sombra) arc (0:180:\sombra)--++(0,-3\unit) arc (180:270:\sombra) --++(\unit,0) arc(270:360:\sombra)--++(0,\unit) arc (0:90:\sombra)--++(-\unit+2\sombra,0) arc (270:180:\sombra)--cycle;
 \node[shift={(0,\sombra)},above,texto] (l) at (c7) {$\N_G(S_1)$};
 
 \filldraw[fill=black!45] (c3) circle (0.7\sombra);
 \node[below right,texto] (l) at (c3) {$S_1$};
 
 \draw (c8)++(0:\sombra) arc (0:180:\sombra)--++(0,-\unit) arc (180:360:\sombra)--cycle;
 \node[texto,shift={(\sombra,-0.5\unit)},right] (l) at (c8) {$G-U-\N_G(S_1)$};
 \foreach \x in{1,2,...,8}
  \node[inode] (\x) at (c\x) {};
 \path[dneg]
 (1) edge[dpos] (2)
 (1) edge (3)
 (2) edge[dpos] (4)
 (3) edge (4)
 (3) edge (5)
 (4) edge (6)
 (5) edge (6)
 (5) edge[dpos] (7)
 (6) edge (8)
 (7) edge (8)
 ;
 \draw[cuadros] (225:1.25\unit) rectangle +(45:2.5\unit);
 \node[texto] (U) at (1.1\unit,0) {$U$};
 \node[texto,right] (l) at (3\unit,0) {$\Fix(G-U-\N_G(S_1))=\Fix\big($};

 \node[inode,minimum size=1pt,shift={(0,-0.5\punit)}] (6) at (l.0) {};
 \node[inode,minimum size=1pt,shift={(0,0.5\punit)}] (8) at (l.0) {};
 \path (6) edge[dneg](8);
 \node[texto,right,shift={(0,0.5\punit)}] (l) at (6) {$\big)=\big\{$};
  \node[inode,minimum size=1pt,shift={(\sep,-0.5\punit)}] (6) at (l.0) {};
 \node[inode,minimum size=1pt,shift={(0,\punit)}] (8) at (6) {};
 \filldraw[fill=black!45] (8) circle (\ssombra);
 \path (6) edge[dneg](8);
 \node[texto,below,yshift=-\sep] (ll) at (6) {$S_{11}$};
 \node[texto,right,shift={(\sep,0)},right] (l) at (6) {$,$};
  \node[inode,minimum size=1pt,shift={(2\sep,0)}] (6) at (l.0) {};
 \node[inode,minimum size=1pt,shift={(0,\punit)}] (8) at (6) {};
 \filldraw[fill=black!45] (6) circle (\ssombra);
 \path (6) edge[dneg](8);
 \node[texto,below,yshift=-\sep] (ll) at (6) {$S_{12}$};
 \node[texto,right,shift={(\sep,0.5\punit)}] (l) at (6) {$\big\}$};
\end{tikzpicture}
\end{center}

The induced subgraph $G-U-\N_G(S_2)$ has also two fixed sets, $S_{21}$ and $S_{22}$: 

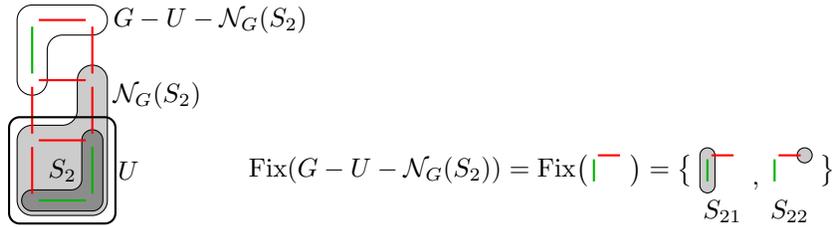
\begin{center}
\begin{tikzpicture}
 \coordinate (c1) at (-0.5\unit,-0.5\unit) ;
 \coordinate (c2) at (0.5\unit,-0.5\unit) ;
 \coordinate (c3) at (-0.5\unit,0.5\unit) ;
 \coordinate (c4) at (0.5\unit,0.5\unit) ;
 \coordinate (c5) at (-0.5\unit,1.5\unit) ;
 \coordinate (c6) at (0.5\unit,1.5\unit) ;
 \coordinate (c7) at (-0.5\unit,2.5\unit) ;
 \coordinate (c8) at (0.5\unit,2.5\unit) ;
 \filldraw[fill=black!20] (c6)++(0:\sombra) arc (0:180:\sombra) --++(0,-\unit+2\sombra) arc(360:270:\sombra)--++(-\unit+2\sombra,0) arc (90:180:\sombra) --++(0,-\unit) arc (180:270:\sombra) --++(\unit,0) arc (270:360:\sombra)--cycle;
 \node[texto,right] (l) at (0.5\unit+\sombra,1.25\unit) {$\N_G(S_2)$};
 
 \filldraw[fill=black!45] (c4)++(0:0.7\sombra) arc (0:180:0.7\sombra) --++(0,-\unit+1.4\sombra) arc(360:270:0.7\sombra)--++(-\unit+1.4\sombra,0) arc (90:270:0.7\sombra)--++(\unit,0) arc(270:360:0.7\sombra)--cycle;
 \node[texto,shift={(-0.5\unit,0.5\unit)}] (l) at (c2) {$S_2$};
 
 \draw (c5)++(180:\sombra) arc (180:360:\sombra) --++(0,\unit-2\sombra) arc(180:90:\sombra)--++(\unit-2\sombra,0) arc (-90:90:\sombra)--++(-\unit,0) arc(90:180:\sombra)--cycle;
 \node[texto,right,shift={(\sombra,0)}] (l) at (c8) {$G-U-\N_G(S_2)$};
 \foreach \x in{1,2,...,8}
  \node[inode] (\x) at (c\x) {};
 \path[dneg]
 (1) edge[dpos] (2)
 (1) edge (3)
 (2) edge[dpos] (4)
 (3) edge (4)
 (3) edge (5)
 (4) edge (6)
 (5) edge (6)
 (5) edge[dpos] (7)
 (6) edge (8)
 (7) edge (8)
 ;
 \draw[cuadros] (225:1.25\unit) rectangle +(45:2.5\unit);
 \node[texto] (U) at (1.1\unit,0) {$U$};
 \node[texto,right] (l) at (3\unit,0) {$\Fix(G-U-\N_G(S_2))=\Fix\big($};

 \node[inode,minimum size=1pt,shift={(0,-0.5\punit)}] (5) at (l.0) {};
 \node[inode,minimum size=1pt,shift={(0,0.5\punit)}] (7) at (l.0) {};
 \node[inode,minimum size=1pt,shift={(\punit,0.5\punit)}] (8) at (l.0) {};
 \path (5) edge[dpos] (7) (7) edge[dneg] (8);

  \node[texto,right,shift={(0,-0.5\punit)}] (l) at (8) {$\big)=\big\{$};
  \node[inode,minimum size=1pt,shift={(\sep,-0.5\punit)}] (5) at (l.0)  {};
 \node[inode,minimum size=1pt,shift={(0,\punit)}] (7) at (5) {};
 \node[inode,minimum size=1pt,shift={(\punit,0)}] (8) at (7) {};
 \filldraw[fill=black!20] (7)++(0:\ssombra) arc (0:180:\ssombra)--++(0,-\punit) arc (180:360:\ssombra)--cycle;
 \node[texto,below,shift={(0.5\punit,-\sep)}] (ll) at (5) {$S_{21}$};
\path (5) edge[dpos] (7) (7) edge[dneg] (8);
 \node[texto,right,shift={(\sep,-\punit)},right] (l) at (8) {$,$};
  \node[inode,minimum size=1pt,shift={(\sep,0)}] (5) at (l.0) {};
 \node[inode,minimum size=1pt,shift={(0,\punit)}] (7) at (5) {};
 \node[inode,minimum size=1pt,shift={(\punit,0)}] (8) at (7) {};
 \filldraw[fill=black!20] (8) circle (\ssombra);
 \node[texto,below,shift={(0.5\punit,-\sep)}] (ll) at (5) {$S_{22}$};
 \path (5) edge[dpos] (7) (7) edge[dneg] (8);
\node[texto,right,shift={(\sep,-0.5\punit)}] (l) at (8) {$\big\}$};

\end{tikzpicture}
\end{center}

Thus $\Fix(G,U)$ contains the following 4 sets:

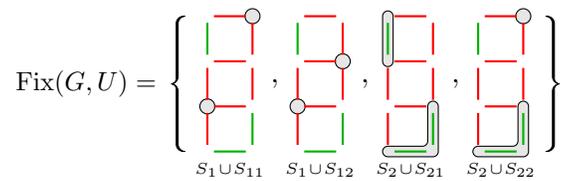
\begin{center}
\begin{tikzpicture}[node distance=\munit]
 \coordinate (c1) at (-0.5\munit,-0.5\munit) ;
 \coordinate[right of = c1] (c2);
 \coordinate[above of = c1] (c3);
 \coordinate[above of = c2] (c4);
 \coordinate[above of = c3] (c5);
 \coordinate[above of = c4] (c6);
 \coordinate[above of = c5] (c7);
 \coordinate[above of = c6] (c8);
 \filldraw[claro] (c3) circle (\ssombra);
 \filldraw[claro] (c8) circle (\ssombra);
 \filldraw[claro] (c8)++(6\munit,0) circle (\ssombra);
 \filldraw[claro] (c3)++(2\munit,0) circle (\ssombra);
 \filldraw[claro] (c6)++(2\munit,0) circle (\ssombra);
 \filldraw[claro] (c4)++(4\munit+0.7\ssombra,0) arc (0:180:0.7\ssombra) --++(0,-\munit+1.4\ssombra) arc(360:270:0.7\ssombra)--++(-\munit+1.4\ssombra,0) arc (90:270:0.7\ssombra)--++(\munit,0) arc(270:360:0.7\ssombra)--cycle;
 \filldraw[claro] (c4)++(6\munit+0.7\ssombra,0) arc (0:180:0.7\ssombra) --++(0,-\munit+1.4\ssombra) arc(360:270:0.7\ssombra)--++(-\munit+1.4\ssombra,0) arc (90:270:0.7\ssombra)--++(\munit,0) arc(270:360:0.7\ssombra)--cycle;
 \filldraw[claro] (c7)++(4\munit+0.7\ssombra,0) arc (0:180:0.7\ssombra)--++(0,-\munit) arc (180:360:0.7\ssombra)--cycle;
\foreach \de in{0,1,2,3}{
    \foreach \x in{1,2,...,8}
      \node[inode,shift={(2*\de\munit,0)}] (\x) at (c\x) {};
    \path[dneg]
      (1) edge[dpos] (2)
      (1) edge (3)
      (2) edge[dpos] (4)
      (3) edge (4)
      (3) edge (5)
      (4) edge (6)
      (5) edge (6)
      (5) edge[dpos] (7)
      (6) edge (8)
      (7) edge (8)
    ;
 }
 \foreach \de in {0,1,2}
  \node(c) at (2*\de\munit+\munit,\munit) {$,$};
  \llaveI{(-\munit,-0.5\munit)}{($(-\munit,-0.75\munit)+(0,3.25\munit)$)}
  \llaveD{(7\munit,-0.5\munit)}{($(7\munit,-0.75\munit)+(0,3.25\munit)$)}
  \node[texto,left](c) at (-1.5\munit,\munit) {$\Fix(G,U)=$};
  \node[texto,below,shift={(-7\munit,-2pt)}](l) at (c1) {\tiny$S_1\!\cup\! S_{11}$};
  \node[texto,shift={(2\munit,0)}](l) at (l) {\tiny$S_1\!\cup\! S_{12}$};
  \node[texto,shift={(2\munit,0)}](l) at (l) {\tiny$S_2\!\cup\! S_{21}$};
  \node[texto,shift={(2\munit,0)}](l) at (l) {\tiny$S_2\!\cup\! S_{22}$};
\end{tikzpicture}
\end{center}

These 4 sets are all fixed sets, since

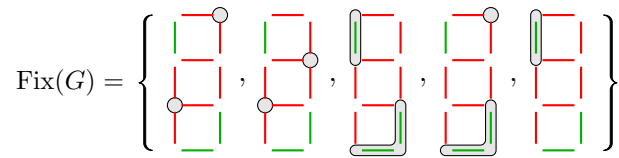
\begin{center}
\begin{tikzpicture}[node distance=\munit]
 \coordinate (c1) at (-0.5\munit,-0.5\munit) ;
 \coordinate[right of = c1] (c2);
 \coordinate[above of = c1] (c3);
 \coordinate[above of = c2] (c4);
 \coordinate[above of = c3] (c5);
 \coordinate[above of = c4] (c6);
 \coordinate[above of = c5] (c7);
 \coordinate[above of = c6] (c8);
 \filldraw[claro] (c3) circle (\ssombra);
 \filldraw[claro] (c8) circle (\ssombra);
 \filldraw[claro] (c8)++(6\munit,0) circle (\ssombra);
 \filldraw[claro] (c3)++(2\munit,0) circle (\ssombra);
 \filldraw[claro] (c6)++(2\munit,0) circle (\ssombra);
 \filldraw[claro] (c4)++(4\munit+0.7\ssombra,0) arc (0:180:0.7\ssombra) --++(0,-\munit+1.4\ssombra) arc(360:270:0.7\ssombra)--++(-\munit+1.4\ssombra,0) arc (90:270:0.7\ssombra)--++(\munit,0) arc(270:360:0.7\ssombra)--cycle;
 \filldraw[claro] (c4)++(6\munit+0.7\ssombra,0) arc (0:180:0.7\ssombra) --++(0,-\munit+1.4\ssombra) arc(360:270:0.7\ssombra)--++(-\munit+1.4\ssombra,0) arc (90:270:0.7\ssombra)--++(\munit,0) arc(270:360:0.7\ssombra)--cycle;
 \filldraw[claro] (c7)++(4\munit+0.7\ssombra,0) arc (0:180:0.7\ssombra)--++(0,-\munit) arc (180:360:0.7\ssombra)--cycle;
 \filldraw[claro] (c7)++(8\munit+0.7\ssombra,0) arc (0:180:0.7\ssombra)--++(0,-\munit) arc (180:360:0.7\ssombra)--cycle;
\foreach \de in{0,1,2,3,4}{
    \foreach \x in{1,2,...,8}
      \node[inode,shift={(2*\de\munit,0)}] (\x) at (c\x) {};
    \path[dneg]
      (1) edge[dpos] (2)
      (1) edge (3)
      (2) edge[dpos] (4)
      (3) edge (4)
      (3) edge (5)
      (4) edge (6)
      (5) edge (6)
      (5) edge[dpos] (7)
      (6) edge (8)
      (7) edge (8)
    ;
 }
 \foreach \de in {0,1,2,3}
  \node(c) at (2*\de\munit+\munit,\munit) {$,$};
  \llaveI{(-\munit,-0.5\munit)}{($(-\munit,-0.75\munit)+(0,3.25\munit)$)}
  \llaveD{(9\munit,-0.5\munit)}{($(9\munit,-0.75\munit)+(0,3.25\munit)$)}
  \node[texto,left](c) at (-1.5\munit,\munit) {$\Fix(G)=$};
\end{tikzpicture}
\end{center}

\end{example}

\begin{lemma}\label{lemma:decompo1}
If $U\subseteq V$ and $G$ has no positive edge between $U$ and $V-U$ then 
\[
\Fix(G,U)\subseteq\Fix(G).
\] 
\end{lemma}

\begin{proof}
Let $S\in \Fix(G[U])$. Let $U'=V(G)-U-\N_G(S)$ and let $S'\in\Fix(G[U'])$. We want to prove that $S\cup S'\in\Fix(G)$. Let $f$, $f^{U}$ and $f^{U'}$ be the conjunctive networks on $G$, $G[U]$ and $G[U']$, respectively. Let $x\in\B^V$ be such that $\ONE(x)=S\cup S'$ and let us prove that $x$ is a fixed point of $f$. Note that $x_{|U}$ is a fixed point of $f^{U}$ and $x_{|U'}$ is a fixed point of $f^{U'}$. 

Suppose first that $x_v=0$. There are three possibilities. First, if $v\in U$ then $f_v(x)\leq f^{U}_v(x_{|U})=x_v=0$ thus $f_v(x)=0$. Second, if $v\in U'$ then $f_v(x)\leq f^{U'}_v(x_{|U'})=x_v=0$ thus $f_v(x)=0$. Finally, suppose that $v\in\N_G(S)-U$. If there is an edge $uv$ with $u\in S$ then $x_u=1$ and since there is no positive edge between $U$ and $V-U$, this edge is negative thus $f_v(x)=0$. Otherwise, by definition of $\N_G(S)$, $v$ belongs to a non-trivial component $C$ of $G^+$ such that $C\subseteq \N_G(S)$. Since $G$ has no positive edge between $U$ and $V-U$ we have $C\cap U=\emptyset$, thus $C\subseteq \N_G(S)-U$. Thus $G$ has positive edge $uv$ with $u\in \N_G(S)-U$. So $x_u=0$ and we deduce that $f_v(x)=0$. Hence, in every case, $f_v(x)=0=x_v$.   

Suppose now that $x_v=1$ and $v\in S$. If $f_v(x)=0$ then $G$ has a positive edge $uv$ with $x_u=0$ or a negative edge $uv$ with $x_u=1$. If $u\in U$ then $f^{U}_v(x_{|U})=0\neq x_v$, a contradiction. Thus $u\not\in U$ and we deduce from the condition of the statement that $uv$ is a negative edge with $x_u=1$. Thus $u\in S'\subseteq U'$ and we have a contradiction with the fact that $u\in N_G(v)\subseteq \N_G(S)$. Therefore $f_v(x)=1=x_v$. 

Suppose finally that $x_v=1$ and $v\in S'$. If $f_v(x)=0$ then $G$ has a positive edge $uv$ with $x_u=0$ or a negative edge $uv$ with $x_u=1$. If $u\in U'$ then $f^{U'}_v(x_{|U'})=0\neq x_v$, a contradiction.  Thus $u\not\in U'$. Since there is no positive edge between $U$ and $V-U$, and no positive edge between $\N_G(S)$ and $V-\N_G(S)$, we deduce that there is no positive edge between $U'$ and $V-U'$. Thus $uv$ is a negative edge with $x_u=1$. Hence, $u\in S\subseteq U$ and we deduce that $v\in N_{G}(u)\subseteq N_G(S)$, a contradiction. Thus $f_v(x)=1=x_v$.  
\end{proof}

The following is an immediate consequence. 

\begin{lemma}\label{lemma:decompo2}
If $U\subseteq V$ and $G$ has no positive edge between $U$ and $V-U$ then 
\[
\fix(G[U])\leq \fix(G,U)\leq \fix(G)
\] 
\end{lemma}

\begin{remark}
We deduce that $\fix(G)\geq 1$ for every simple signed graph $G$. Indeed, if $G$ has no positive edge  then $\fix(G)=\mis(G)\geq 1$. Otherwise, $G^+$ has a connected component $C$, and by Lemma~\ref{lemma:decompo2} and Proposition~\ref{pro:basic1} we have $\fix(G)\geq\fix(G[U])\geq 1$.  
\end{remark}

We now prove a lemma with that gives a stronger conclusion under stronger conditions. 

\begin{lemma}\label{lemma:decompo3}
If $U\subseteq V$, if $G$ has no positive edge between $U$ and $V-U$, and if every vertex in $U$ is adjacent to a positive edge, then 
\[
\Fix(G,U)=\Fix(G).
\] 
\end{lemma}

\begin{proof}
By Lemma~\ref{lemma:decompo1}, it is sufficient to prove that $\Fix(G)\subseteq \Fix(G,U)$. Let $f$ be the conjunctive network on $G$, and let $x$ be a fixed point of $f$. Let 
\[
S=\ONE(x)\cap U,\qquad
U'=V-U-\N_G(S),\qquad
S'=\ONE(x)\cap U'.
\]
Let $f^{U}$ and $f^{U'}$ be the conjunctive networks on $G[U]$ and $G[U']$.

Let $v\in U$. Since $v$ is adjacent to a positive edge and since $G$ has no positive edge between $U$ and $V-U$, we deduce that $G[U]$ has a positive edge $uv$. If $x_v=0$ then $x_u=0$ and we deduce that $f^{U}_v(x_{|U})=0=x_v$. If $x_v=1$ then $f_v(x)=1\leq f^{U}_v(x_{|U})$ thus $f^{U}_v(x_{|U})=1=x_v$. Hence, we have proved that $x_{|U}$ is a fixed point of $f^{U}$, that is, 
\begin{equation}\label{decompo3.1}
S\in \Fix(G[U]).
\end{equation}

Let $v\in \N_G(S)-U$ and let us prove that $x_v=0$. By the definition of $\N_G(S)$, $v$ belongs to a component $C$ of $G^+$ such that $C\cap N_G(S)\neq\emptyset$ and $C\subseteq \N_G(S)$. Actually, $C\subseteq \N_G(S)-U$ since otherwise $G$ has a positive edge between $U$ and $V-U$. Thus there exists a path $u,w_1,w_2,\dots,w_k$ from $u\in S$ to $w_k=v$, where all the vertices $w_\ell$ are in $C$. Since $uw_1$ is negative and $x_u=1$ we have $f_{w_1}(x)=0$ thus $x_{w_1}=0$ and we deduce from Proposition~\ref{pro:basic1} that $x_v=0$. Thus we have  $\ONE(x)\cap (\N_G(S)-U)=\emptyset$. In other words, 
\begin{equation}\label{decompo3.2}
\ONE(x)=S\cup S'.
\end{equation}
 
Let $v\in U'$. Suppose that $x_v=0=f_v(x)$, and suppose, for a contradiction that $f^{U'}_v(x_{|U'})=1$. Since there is no positive edge between $U$ and $V-U$, and no positive edge between $\N_G(S)$ and $V-\N_G(S)$, we deduce that there is no positive edge between $U'$ and $V-U'$. Thus there exists a negative edge $uv$ with $u\in V-U'$ and $x_u=1$. But then $u\in S$, thus $v\in \N_G(S)$, a contradiction. Consequently, $f^{U'}_v(x_{|U'})=0=x_v$. Now, if $x_v=1$ then $f_v(x)=1\leq f^{U'}_v(x_{|U'})$ thus $f^{U'}_v(x_{|U'})=1=x_v$. Hence, we have proved that $x_{|U'}$ is a fixed point of $f^{U'}$, that is, 
\begin{equation}\label{decompo3.3}
S'\in \Fix(G-U-\N_G(S)).
\end{equation} 

According to \eqref{decompo3.1}, \eqref{decompo3.2} and \eqref{decompo3.3}, we have $\ONE(x)\in\Fix(G,U)$. 
\end{proof}

In the following, if $S$ is a subset of $V$ and $\mathcal{S}$ is a subset of the power set of $V$, then
\[
S\sqcup\mathcal{S}=\{S\cup S'\,|\,S'\in\mathcal{S}\}.
\]

\begin{lemma}\label{lemma:decompo4}
Let $C$ be is a non-trivial connected component of $G^+$. If $G[C]$ has no negative edge then 
\[
\Fix(G)=\Fix(G-C) \cup \big(C\sqcup\Fix(G-\N_G(C))\big)
\]
and otherwise 
\[
\Fix(G)=\Fix(G-C).
\]
\end{lemma}

\begin{proof}
If $G[C]$ has no negative edge then by Proposition~\ref{pro:basic1} we have $\Fix(G[C])=\{\emptyset,C\}$ and from Lemma~\ref{lemma:decompo3} we deduce that  
\begin{align*}
\Fix(G)&=\Fix(G,C)\\
&=\{S\cup S'\,|\,S\in\{\emptyset,C\},~S'\in\Fix(G-C-\N_G(S))\}\\
&=\{\emptyset\cup S'\,|\,S'\in \Fix(G-C)\}\cup \{C\cup S'\,|\,S'\in \Fix(G-\N_G(C))\}\\
&= \Fix(G-C) \cup \big(C\sqcup\Fix(G-\N_G(C))\big).
\end{align*}
If $G[C]$ has a negative edge then by Proposition~\ref{pro:basic1} we have $\Fix(G[C])=\{\emptyset\}$ and proceeding as above we get $\Fix(G)=\Fix(G-C)$. 
\end{proof}

The proof of Lemma~\ref{lemma:matching} is a straightforward application of the above decomposition~tools. 

\begin{proof}[Proof of Lemma~\ref{lemma:matching}]
Suppose that $G^+$ has a non-trivial component $C$ such that $G[C]$ has only positive edges. Let $ab$ be one of these edges, and let $G'$ be the signed graph obtained from $G$ by making negative each edge of $G[C]$ excepted $ab$. To prove that lemma, it is clearly sufficient to prove that $\fix(G')\geq \fix(G)$ (because then, the reduction process of positive components in single positive edges can be repeated until that each positive component of $G^+$ reduces to a single positive edge). According to Lemma \ref{lemma:decompo4} we have 
\begin{equation}\label{positive2.1}
\fix(G)\leq \fix(G-C)+\fix(G-\N_{G}(C))
\end{equation}
and
\begin{equation}\label{positive2.2}
\fix(G')=\fix(G'-\{a,b\})+\fix(G'-\N_{G'}(\{a,b\})).
\end{equation}
Since $G-C=G'-C$ and $\{a,b\}\subseteq C$, $G-C$ is an induced subgraph of $G'-\{a,b\}$, thus according to Lemma~\ref{lemma:decompo2} we have  
\begin{equation}\label{positive2.3}
\fix(G-C)\leq \fix(G'-\{a,b\}).
\end{equation}
Since $\{a,b\}\subseteq C$ and since each positive edge of $G'$ is a positive edge of $G$, we have $\N_{G'}(\{a,b\})\subseteq \N_G(C)$. Since $G-C=G'-C$ we deduce that $G-\N_{G}(C)$ is an induced subgraph of $G'-\N_{G'}(\{a,b\})$. Thus according to Lemma~\ref{lemma:decompo2} we have 
\begin{equation}\label{positive2.4}
\fix(G-\N_{G}(C))\leq \fix(G'-\N_{G'}(\{a,b\})).
\end{equation}
The lemma follows from \eqref{positive2.1}, \eqref{positive2.2}, \eqref{positive2.3} and  \eqref{positive2.4}.
\end{proof}

\subsection{Proof of Lemma~\ref{lemma:C_4} (suppression of positive edges)}\label{sec:suppression}

In all this section, $G$ is a simple signed graph with vertex set $V$ in which the set of positive edges is a matching. Let $G'$ be any simple signed graph obtained from $G$ by making negative a positive edge $e$. We will study the variation of the number of fixed sets under this transformation. Actually, we will prove that if $G$ has no induced copy of $C_4$ that contains $e$ and no other positive edge, then $\fix(G)\leq \fix(G')$, and in any case $\fix(G)\leq (3/2)\,\fix(G')$. The control of the variation of fixed sets is a little bit technical. This is why we decompose the proof in several steps. We first assume that (i) there is a unique positive edge in $G$ and that (ii) this positive edge is in the neighborhood of each vertex (cf. Lemma~\ref{lemma:edge1}). Then, thanks to an additional decomposition property (cf. Lemma~\ref{lemma:edge2}), we suppress condition (ii). We finally suppress condition (i) to get the general statement (cf. Lemma~\ref{lemma:edge3}). 

\begin{lemma}\label{lemma:edge1}
Suppose that $G$ has a unique positive edge, say $ab$, and $V=N_G(a)\cup N_G(b)$. Let $G'$ be the simple signed graph obtained from $G$ by making $ab$ negative. Then 
\[
\fix(G)\leq \fix(G')+1, 
\]
and if $G$ has no induced copy of $C_4$ containing $ab$, then 
\[
\fix(G)\leq \fix(G').
\]
\end{lemma}

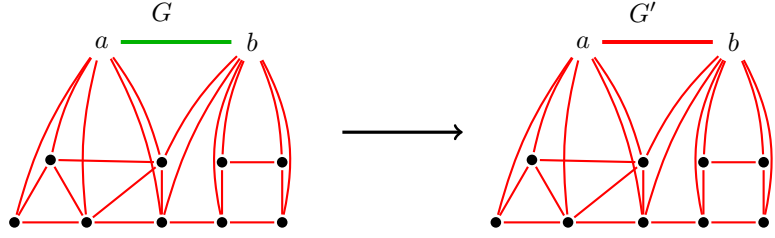
\begin{figure}[!htb]
\centering
\begin{tikzpicture}
  \node[texto] (l)at (0,0.5\unit){$G$};
 \node[inode,texto](u) at (-\unit,0){$a$};
 \node[inode,texto](v) at (1.5\unit,0){$b$};
 \node[vnode](1) at (0,-3\unit){};
 \node[vnode](2) at (0,-2\unit){};
 \node[vnode] (3) at (-1.25\unit,-3\unit){};
 \node[vnode,shift={(180:1.2\unit)}] (4) at (3){};
 \node[vnode,shift={(120:1.2\unit)}] (5) at (3){};
 \node[vnode] (6) at (1\unit,-3\unit){};
 \node[vnode] (7) at (2\unit,-3\unit){};
 \node[vnode] (8) at (1\unit,-2\unit){};
 \node[vnode] (9) at (2\unit,-2\unit){};
 \path[dneg]
 (u) edge[dpos,grueso] (v)
 (u) edge[bend left=10] (1)
 (u) edge[bend left=10] (2)
 (u) edge[bend right=10] (3)
 (u) edge[bend right=10] (4)
 (u) edge[bend right=10] (5)
 (v) edge[bend right=10] (1)
 (v) edge[bend right=10] (2)
 (v) edge[bend right=20] (6)
 (v) edge[bend left=20] (7)
 (v) edge[bend right=10] (8)
 (v) edge[bend left=10] (9)
 (1) edge (2)
 (1) edge (3)
 (5) edge (2)
 (3) edge (2)
 (3) edge (4)
 (3) edge (5)
 (5) edge (4)
 (1) edge (6)
 (7) edge (6)
 (8) edge (6)
 (7) edge (9)
 (8) edge (9)
 ;
 \flecha{(3\unit,-1.5\unit)}{2\unit}
 
 \pgftransformxshift{8\unit}
 \node[texto] (l)at (0,0.5\unit){$G'$};
 \node[inode,texto](u) at (-\unit,0){$a$};
 \node[inode,texto](v) at (1.5\unit,0){$b$};
 \node[vnode](1) at (0,-3\unit){};
 \node[vnode](2) at (0,-2\unit){};
 \node[vnode] (3) at (-1.25\unit,-3\unit){};
 \node[vnode,shift={(180:1.2*\unit)}] (4) at (3){};
 \node[vnode,shift={(120:1.2*\unit)}] (5) at (3){};
 \node[vnode] (6) at (1\unit,-3\unit){};
 \node[vnode] (7) at (2\unit,-3\unit){};
 \node[vnode] (8) at (1\unit,-2\unit){};
 \node[vnode] (9) at (2\unit,-2\unit){};
 \path[dneg]
 (u) edge[grueso] (v)
 (u) edge[bend left=10] (1)
 (u) edge[bend left=10] (2)
 (u) edge[bend right=10] (3)
 (u) edge[bend right=10] (4)
 (u) edge[bend right=10] (5)
 (v) edge[bend right=10] (1)
 (v) edge[bend right=10] (2)
 (v) edge[bend right=20] (6)
 (v) edge[bend left=20] (7)
 (v) edge[bend right=10] (8)
 (v) edge[bend left=10] (9)
 (1) edge (2)
 (1) edge (3)
 (5) edge (2)
 (3) edge (2)
 (3) edge (4)
 (3) edge (5)
 (5) edge (4)
 (1) edge (6)
 (7) edge (6)
 (8) edge (6)
 (7) edge (9)
 (8) edge (9)
 ;

\end{tikzpicture}
\caption{Illustration of the transformation described in Lemma~\ref{lemma:edge1}.}\label{fig:edge1}
\end{figure}

\begin{proof}
According to Lemma \ref{lemma:decompo4} we have 
\[
\Fix(G)= \Fix(G-\{a,b\}) \cup 
\big(\{a,b\}\sqcup \Fix(G-\N_G(\{a,b\})\big)
\]
Since $V=N_G(\{a,b\})\subseteq \N_G(\{a,b\})$, $G-\N_G(\{a,b\})$ is empty, and since $G-\{a,b\}$ has only negative edges we deduce that 
\begin{equation}\label{edge1.1}
\fix(G)=\mis(G-\{a,b\})+1.
\end{equation}
Also, since $G'$ has only negative edges (see Figure~\ref{fig:edge1}) we have, 
\begin{equation}\label{edge1.2}
\fix(G')=\mis(G).
\end{equation}
Let $A=N_G(a)-N_G(b)$, $B=N_G(b)-N_G(a)$ and $C=N_G(a)\cap N_G(b)$ (see Figure~\ref{fig:edge1_0}). Note that $G$ has an induced copy of $C_4$ containing $ab$ if and only if there is an edge between $A$ and $B$. We consider four cases. 

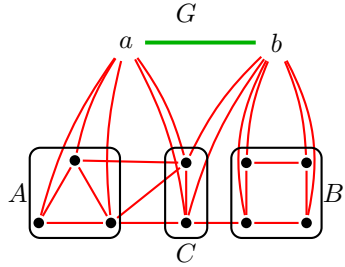
\begin{figure}[!htb]
\centering
\begin{tikzpicture}
  \node[texto] (l)at (0,0.5\unit){$G$};
 \node[inode,texto](u) at (-\unit,0){$a$};
 \node[inode,texto](v) at (1.5\unit,0){$b$};
 \node[vnode](1) at (0,-3\unit){};
 \node[vnode](2) at (0,-2\unit){};
 \node[vnode] (3) at (-1.25\unit,-3\unit){};
 \node[vnode,shift={(180:1.2\unit)}] (4) at (3){};
 \node[vnode,shift={(120:1.2\unit)}] (5) at (3){};
 \node[vnode] (6) at (1\unit,-3\unit){};
 \node[vnode] (7) at (2\unit,-3\unit){};
 \node[vnode] (8) at (1\unit,-2\unit){};
 \node[vnode] (9) at (2\unit,-2\unit){};
 \path[dneg]
 (u) edge[dpos,grueso] (v)
 (u) edge[bend left=10] (1)
 (u) edge[bend left=10] (2)
 (u) edge[bend right=10] (3)
 (u) edge[bend right=10] (4)
 (u) edge[bend right=10] (5)
 (v) edge[bend right=10] (1)
 (v) edge[bend right=10] (2)
 (v) edge[bend right=20] (6)
 (v) edge[bend left=20] (7)
 (v) edge[bend right=10] (8)
 (v) edge[bend left=10] (9)
 (1) edge (2)
 (1) edge (3)
 (5) edge (2)
 (3) edge (2)
 (3) edge (4)
 (3) edge (5)
 (5) edge (4)
 (1) edge (6)
 (7) edge (6)
 (8) edge (6)
 (7) edge (9)
 (8) edge (9)
 ;
 \draw[cuadros] (-2.6\unit,-3.25\unit)rectangle+(1.5\unit,1.5\unit);
 \node[texto](l)at (-2.8\unit,-2.5\unit){$A$};
 \draw[cuadros] (-0.35\unit,-3.25\unit)rectangle+(0.7\unit,1.5\unit);
 \node[texto](l)at (0,-3.5\unit){$C$};
 \draw[cuadros] (0.75\unit,-3.25\unit)rectangle+(1.5\unit,1.5\unit);
 \node[texto](l)at (2.45\unit,-2.5\unit){$B$};
   \end{tikzpicture}
\caption{Illustration of the sets $A$, $B$ and $C$ described in the proof of Lemma~\ref{lemma:edge1}.}\label{fig:edge1_0}
\end{figure}

\begin{enumerate}
\item
Suppose that $A\neq\emptyset$ and $B\neq\emptyset$. Let $S$ be a maximal independent set of $G-\{a,b\}$. If $S$ intersects $C$, or $S$ intersects both $A$ and $B$, then $a$ and $b$ have at least one neighbor in $S$, thus $S$ is a maximal independent set of $G$. Otherwise it is easy to check that either $S$ intersects $A$ and $S\cup\{b\}$ is a maximal independent set of $G$, or $S$ intersects $B$ and $S\cup\{a\}$ is a maximal independent set of $G$. Consequently, the following map is an injection from $\Mis(G-\{a,b\})$ to $\Mis(G)$:
\[
S\mapsto 
\left\{
\begin{array}{llcll}
S&\text{if }S\cap C\neq\emptyset&\text{or}&(S\cap A\neq\emptyset\text{ and }S\cap B\neq\emptyset)&\quad(i)\\
S\cup\{b\}&\text{if }S\cap C=\emptyset&\text{and}&(S\cap A\neq\emptyset\text{ and }S\cap B=\emptyset)&\quad(ii)\\
S\cup\{a\}&\text{if }S\cap C=\emptyset&\text{and}&(S\cap A=\emptyset\text{ and }S\cap B\neq\emptyset)&\quad(iii)\\
\end{array}
\right.
\]
Thus $\mis(G-\{a,b\})\leq\mis(G)$ and we deduce from \eqref{edge1.1} and \eqref{edge1.2} that $\fix(G)\leq \fix(G')+1$. Furthermore, if $G$ has no induced copy of $C_4$ containing $ab$, then there is no edge between $A$ and $B$, and thus cases $(ii)$ and $(iii)$ are not possible. Thus $\Mis(G-\{a,b\})\subseteq \Mis(G)$ and since $G$ has at least one maximal independent set containing $a$ and one maximal independent set containing $b$, we deduce that \mbox{$\mis(G-\{a,b\})\leq \mis(G)-2$} (see Figure~\ref{fig:edge1_1}). Using \eqref{edge1.1} and \eqref{edge1.2} we get $\fix(G)\leq \fix(G')$ (see Figure~\ref{fig:edge1_1}).     

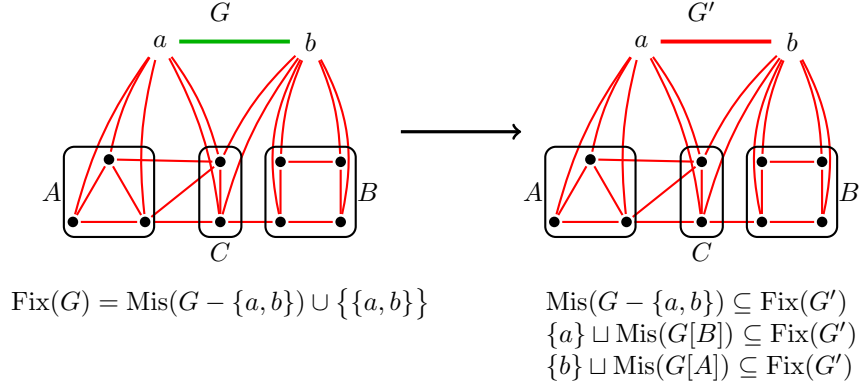
\begin{figure}[!htb]
\centering
\begin{tikzpicture}
 \node[texto] (l)at (0,0.5\unit){$G$};
 \node[inode,texto](u) at (-\unit,0){$a$};
 \node[inode,texto](v) at (1.5\unit,0){$b$};
 \node[vnode](1) at (0,-3\unit){};
 \node[vnode](2) at (0,-2\unit){};
 \node[vnode] (3) at (-1.25\unit,-3\unit){};
 \node[vnode,shift={(180:1.2\unit)}] (4) at (3){};
 \node[vnode,shift={(120:1.2\unit)}] (5) at (3){};
 \node[vnode] (6) at (1\unit,-3\unit){};
 \node[vnode] (7) at (2\unit,-3\unit){};
 \node[vnode] (8) at (1\unit,-2\unit){};
 \node[vnode] (9) at (2\unit,-2\unit){};
 \path[dneg]
 (u) edge[dpos,grueso] (v)
 (u) edge[bend left=10] (1)
 (u) edge[bend left=10] (2)
 (u) edge[bend right=10] (3)
 (u) edge[bend right=10] (4)
 (u) edge[bend right=10] (5)
 (v) edge[bend right=10] (1)
 (v) edge[bend right=10] (2)
 (v) edge[bend right=20] (6)
 (v) edge[bend left=20] (7)
 (v) edge[bend right=10] (8)
 (v) edge[bend left=10] (9)
 (1) edge (2)
 (1) edge (3)
 (5) edge (2)
 (3) edge (2)
 (3) edge (4)
 (3) edge (5)
 (5) edge (4)
 (1) edge (6)
 (7) edge (6)
 (8) edge (6)
 (7) edge (9)
 (8) edge (9)
 ;
  \draw[cuadros] (-2.6\unit,-3.25\unit)rectangle+(1.5\unit,1.5\unit);
 \node[texto](l)at (-2.8\unit,-2.5\unit){$A$};
 \draw[cuadros] (-0.35\unit,-3.25\unit)rectangle+(0.7\unit,1.5\unit);
 \node[texto](l)at (0,-3.5\unit){$C$};
\draw[cuadros] (0.75\unit,-3.25\unit)rectangle+(1.5\unit,1.5\unit);
\node[texto](l)at (2.45\unit,-2.5\unit){$B$};
  \node[texto,below](l)at (0,-4\unit){$\Fix(G)=\Mis(G-\{a,b\})\cup \big\{\{a,b\}\big\}$};

 \flecha{(3\unit,-1.5\unit)}{2\unit}
 
 \pgftransformxshift{8\unit}
 \node[texto] (l)at (0,0.5\unit){$G'$};
 \node[inode,texto](u) at (-\unit,0){$a$};
 \node[inode,texto](v) at (1.5\unit,0){$b$};
 \node[vnode](1) at (0,-3\unit){};
 \node[vnode](2) at (0,-2\unit){};
 \node[vnode] (3) at (-1.25\unit,-3\unit){};
 \node[vnode,shift={(180:1.2\unit)}] (4) at (3){};
 \node[vnode,shift={(120:1.2\unit)}] (5) at (3){};
 \node[vnode] (6) at (1\unit,-3\unit){};
 \node[vnode] (7) at (2\unit,-3\unit){};
 \node[vnode] (8) at (1\unit,-2\unit){};
 \node[vnode] (9) at (2\unit,-2\unit){};
 \path[dneg]
 (u) edge[grueso] (v)
 (u) edge[bend left=10] (1)
 (u) edge[bend left=10] (2)
 (u) edge[bend right=10] (3)
 (u) edge[bend right=10] (4)
 (u) edge[bend right=10] (5)
 (v) edge[bend right=10] (1)
 (v) edge[bend right=10] (2)
 (v) edge[bend right=20] (6)
 (v) edge[bend left=20] (7)
 (v) edge[bend right=10] (8)
 (v) edge[bend left=10] (9)
 (1) edge (2)
 (1) edge (3)
 (5) edge (2)
 (3) edge (2)
 (3) edge (4)
 (3) edge (5)
 (5) edge (4)
 (1) edge (6)
 (7) edge (6)
 (8) edge (6)
 (7) edge (9)
 (8) edge (9)
 ;
 \draw[cuadros] (-2.6\unit,-3.25\unit) rectangle +(1.5\unit,1.5\unit);
 \node[texto] (l) at (-2.8\unit,-2.5\unit){$A$};
  \draw[cuadros] (-0.35\unit,-3.25\unit) rectangle+(0.7\unit,1.5\unit);
  \node[texto] (l) at (0,-3.5\unit){$C$};
 \draw[cuadros] (0.75\unit,-3.25\unit) rectangle+(1.5\unit,1.5\unit);
 \node[texto] (l) at (2.45\unit,-2.5\unit){$B$};
  \node[texto,below](l) at (0,-4\unit){$\begin{array}{l}\Mis(G-\{a,b\})\subseteq \Fix(G')\\ \{a\}\sqcup \Mis(G[B])\subseteq \Fix(G')\\\{b\}\sqcup\Mis(G[A])\subseteq \Fix(G')\end{array}$};
\end{tikzpicture}
\caption{Fixed sets when $A\neq\emptyset$ and $B\neq\emptyset$.}\label{fig:edge1_1}
\end{figure}

\item
Suppose that $A\neq\emptyset$ and $B=\emptyset$. 
Let $S$ be a maximal independent set of $G-\{a,b\}$. If $S$ intersects $C$ then $S$ is clearly a maximal independent set of $G$. Otherwise, $S$ intersects $A$ and $S\cup\{b\}$ is then a maximal independent set of $G$. Thus, the following map is an injection from $\Mis(G-\{a,b\})$ to $\Mis(G)$:
\[
S\mapsto 
\left\{
\begin{array}{ll}
S&\text{if }S\cap C\neq\emptyset\\
S\cup\{b\}&\text{otherwise}
\end{array}
\right.
\]
Since $G$ has at least one maximal independent set containing $a$, we deduce that $\mis(G-\{a,b\})\leq \mis(G)-1$. Using \eqref{edge1.1} and \eqref{edge1.2} we get $\fix(G)\leq\fix(G')$ (see Figure~\ref{fig:edge1_2}).     

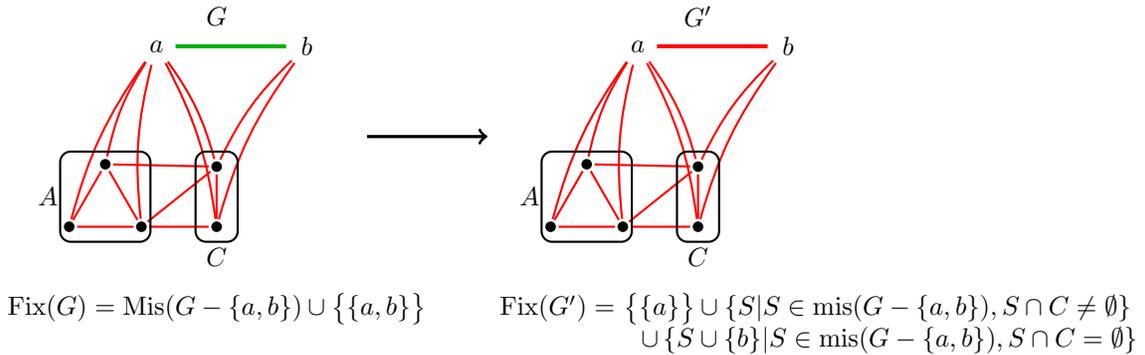
\begin{figure}[!htb]
\centering
\begin{tikzpicture}
  \node[texto] (l)at (0,0.5\unit){$G$};
 \node[inode,texto](u) at (-\unit,0){$a$};
 \node[inode,texto](v) at (1.5\unit,0){$b$};
 \node[vnode](1) at (0,-3\unit){};
 \node[vnode](2) at (0,-2\unit){};
 \node[vnode] (3) at (-1.25\unit,-3\unit){};
 \node[vnode,shift={(180:1.2\unit)}] (4) at (3){};
 \node[vnode,shift={(120:1.2\unit)}] (5) at (3){};
 \path[dneg]
 (u) edge[dpos,grueso] (v)
 (u) edge[bend left=10] (1)
 (u) edge[bend left=10] (2)
 (u) edge[bend right=10] (3)
 (u) edge[bend right=10] (4)
 (u) edge[bend right=10] (5)
 (v) edge[bend right=10] (1)
 (v) edge[bend right=10] (2)
 (1) edge (2)
 (1) edge (3)
 (5) edge (2)
 (3) edge (2)
 (3) edge (4)
 (3) edge (5)
 (5) edge (4)
 ;
  \draw[cuadros] (-2.6\unit,-3.25\unit)rectangle+(1.5\unit,1.5\unit);
 \node[texto](l)at (-2.8\unit,-2.5\unit){$A$};
 \draw[cuadros] (-0.35\unit,-3.25\unit)rectangle+(0.7\unit,1.5\unit);
 \node[texto](l)at (0,-3.5\unit){$C$};
  \node[texto,below](l)at (0,-4\unit){$\Fix(G)=\Mis(G-\{a,b\})\cup \big\{\{a,b\}\big\}$};

 \flecha{(2.5\unit,-1.5\unit)}{2\unit}

 \pgftransformxshift{8\unit}
 \node[texto] (l)at (0,0.5\unit){$G'$};
 \node[inode,texto](u) at (-\unit,0){$a$};
 \node[inode,texto](v) at (1.5\unit,0){$b$};
 \node[vnode](1) at (0,-3\unit){};
 \node[vnode](2) at (0,-2\unit){};
 \node[vnode] (3) at (-1.25\unit,-3\unit){};
 \node[vnode,shift={(180:1.2\unit)}] (4) at (3){};
 \node[vnode,shift={(120:1.2\unit)}] (5) at (3){};
 \path[dneg]
 (u) edge[grueso] (v)
 (u) edge[bend left=10] (1)
 (u) edge[bend left=10] (2)
 (u) edge[bend right=10] (3)
 (u) edge[bend right=10] (4)
 (u) edge[bend right=10] (5)
 (v) edge[bend right=10] (1)
 (v) edge[bend right=10] (2)
 (1) edge (2)
 (1) edge (3)
 (5) edge (2)
 (3) edge (2)
 (3) edge (4)
 (3) edge (5)
 (5) edge (4)
 ;
 \draw[cuadros] (-2.6\unit,-3.25\unit) rectangle +(1.5\unit,1.5\unit);
 \node[texto] (l) at (-2.8\unit,-2.5\unit){$A$};
  \draw[cuadros] (-0.35\unit,-3.25\unit) rectangle+(0.7\unit,1.5\unit);
  \node[texto] (l) at (0,-3.5\unit){$C$};
  \node[texto,below](l) at (1.6,-4\unit){$\begin{array}{ll}\Fix(G')=&\!\!\!\!\big\{\{a\}\big\}\cup \{S|S\in\mis(G-\{a,b\}), S\cap C\neq\emptyset\}\\&\cup\,\{S\cup\{b\}|S\in\mis(G-\{a,b\}), S\cap C=\emptyset\}\end{array}$};
\end{tikzpicture}
\caption{Fixed sets when $A=\emptyset$ or $B=\emptyset$.}\label{fig:edge1_2}
\end{figure}

\item
Suppose that $A=\emptyset$ and $B\neq\emptyset$. We prove as in case 2 that $\fix(G)\leq\fix(G')$. 

\item
Suppose that $A=\emptyset$ and $B=\emptyset$. If $C=\emptyset$ then $G$ reduces to a single positive edge and $G'$ reduced to a single negative edge and thus $\fix(G)=\fix(G')=2$ (see Figure~\ref{fig:edge1_4}). 

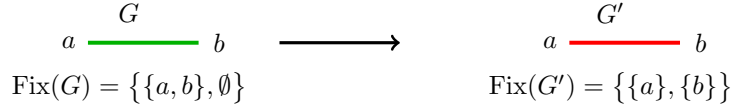
\begin{figure}[!htb]
\centering
\begin{tikzpicture}
  \node[texto] (l)at (0,0.5\unit){$G$};
 \node[inode,texto](u) at (-\unit,0){$a$};
 \node[inode,texto](v) at (1.5\unit,0){$b$};
 \path[dpos,grueso] (u) edge (v) ;
  \node[texto](l)at (0,-0.6){$\Fix(G)=\big\{\{a,b\},\emptyset\big\}$};

 \flecha{(2.5\unit,0)}{2\unit}
 
 \pgftransformxshift{8\unit}
 \node[texto] (l)at (0,0.5\unit){$G'$};
 \node[inode,texto](u) at (-\unit,0){$a$};
 \node[inode,texto](v) at (1.5\unit,0){$b$};
 \path[dneg,grueso] (u) edge (v) ;
 \node[texto] (l)at (0,-0.6){$\Fix(G')=\big\{\{a\},\{b\}\big\}$};
\end{tikzpicture}
\caption{Fixed sets when $A=\emptyset$, $B=\emptyset$ and $C=\emptyset$.}\label{fig:edge1_4}
\end{figure}

So suppose that $C\neq\emptyset$. Then $G-\{a,b\}=G[C]$ and it is clear that $\Mis(G-\{a,b\})\subseteq \Mis(G)$. As in the first case, since $G$ has at least one maximal independent set containing $a$ and one maximal independent set containing $b$, we deduce that $\mis(G-\{a,b\})\leq\mis(G)-2$. Using \eqref{edge1.1} and \eqref{edge1.2} we get $\fix(G)\leq \fix(G')$ (see Figure~\ref{fig:edge1_3}). 

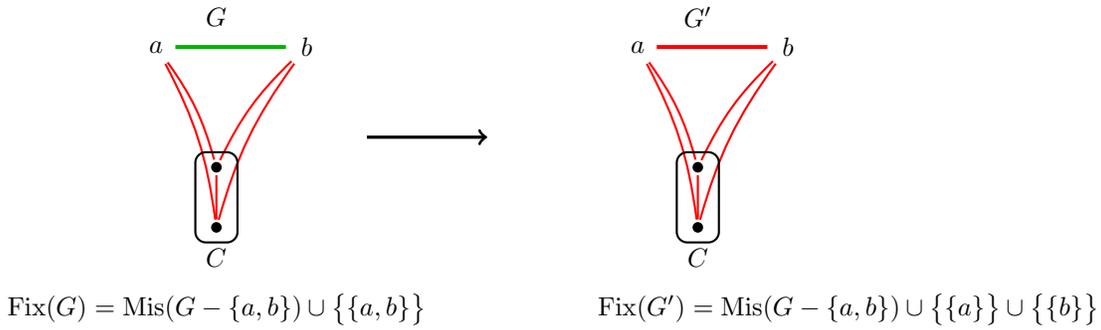
\begin{figure}[!htb]
\centering
\begin{tikzpicture}
  \node[texto] (l)at (0,0.5\unit){$G$};
 \node[inode,texto](u) at (-\unit,0){$a$};
 \node[inode,texto](v) at (1.5\unit,0){$b$};
 \node[vnode](1) at (0,-3\unit){};
 \node[vnode](2) at (0,-2\unit){};
 \path[dneg]
 (u) edge[dpos,grueso] (v)
 (u) edge[bend left=10] (1)
 (u) edge[bend left=10] (2)
 (v) edge[bend right=10] (1)
 (v) edge[bend right=10] (2)
 (1) edge (2)
 ;
 \draw[cuadros] (-0.35\unit,-3.25\unit)rectangle+(0.7\unit,1.5\unit);
 \node[texto](l)at (0,-3.5\unit){$C$};
  \node[texto,below](l)at (0,-4\unit){$\Fix(G)=\Mis(G-\{a,b\})\cup \big\{\{a,b\}\big\}$};

 \flecha{(2.5\unit,-1.5\unit)}{2\unit}
 
 \pgftransformxshift{8\unit}
 \node[texto] (l)at (0,0.5\unit){$G'$};
 \node[inode,texto](u) at (-\unit,0){$a$};
 \node[inode,texto](v) at (1.5\unit,0){$b$};
 \node[vnode](1) at (0,-3\unit){};
 \node[vnode](2) at (0,-2\unit){};
 \path[dneg]
 (u) edge[grueso] (v)
 (u) edge[bend left=10] (1)
 (u) edge[bend left=10] (2)
 (v) edge[bend right=10] (1)
 (v) edge[bend right=10] (2)
 (1) edge (2)
 ;
  \draw[cuadros] (-0.35\unit,-3.25\unit) rectangle+(0.7\unit,1.5\unit);
  \node[texto] (l) at (0,-3.5\unit){$C$};
  \node[texto,below] (l)at (2,-4\unit){$\Fix(G')=\Mis(G-\{a,b\})\cup \big\{\{a\}\big\}\cup \big\{\{b\}\big\}$};
\end{tikzpicture}
\caption{Fixed sets when $A=\emptyset$, $B=\emptyset$ and $C\neq\emptyset$.}\label{fig:edge1_3}
\end{figure}

\end{enumerate}
\end{proof}

The next decomposition property is a rather technical step which will allows us to suppress the condition $V=N_G(a)\cup N_G(b)$ from the previous proposition.  

\begin{lemma}\label{lemma:edge2}
Suppose that $G$ has a positive edge $ab$ such that $a$ and $b$ are not adjacent to other positive edges. Let $X$ be the set of vertices containing $a$, $b$ and all the vertices of $N_G(a)\cup N_G(b)$ adjacent to only negative edges. Let $G'$ be the simple signed graph obtained from $G$ by making $ab$ negative, and let $U=V-X$. Then 
\[
\fix(G)-\fix(G,U)\leq \fix(G')-\fix(G',U).
\] 
\end{lemma}

\begin{proof} 
Let $x\in\Fix(G)-\Fix(G,U)$. Let $S=\ONE(x)\cap U$ and $U'=V-U-\N_G(S)$. Since $\ONE(x)$ and $\N_G(S)-S$ are disjointed, we have  $\ONE(x)=S\cup S'$ with $S'=\ONE(x)\cap U'$. Let $f$, $f^U$ and $f^{U'}$ be the conjunctive networks on $G$, $G[U]$ and $G[U']$ respectively.  

\begin{claim}
$S\not\in\Fix(G[U])$.
\end{claim}

\begin{subproof}
Since $S\cup S'\not\in \Fix(G,U)$, it is sufficient to prove that $S'\in\Fix(G[U'])$, which is equivalent to prove that $x_{|U'}$ is a fixed point of $f^{U'}$. Let $v\in U'$. Since $f_v(x)\leq f^{U'}_v(x_{|U'})$ if $x_v=1$ then $f^{U'}_v(x)=1$. So suppose that $x_v=0$. Then $G$ has a negative edge $uv$ with $x_u=1$ or a positive edge $uv$ with $x_u=0$. Suppose first that $G$ has a negative edge $uv$ with $x_u=1$. If $u\in S$ then $v\in N_G(S)$ and this is not possible since $v\in U'$. Thus $u\in S'\subseteq U'$ and we deduce that $f^{U'}_v(x_{|U'})=0$. Suppose now that $G$ has a positive edge $uv$ with $x_u=0$. Since there is no positive edge between $U$ and $V-U$ and no positive edge between $\N_G(S)$ and $V-\N_G(S)$, there is no positive edges between $U'$ and $V-U'$. Therefore $u\in U'$ and we deduce that $f^{U'}_v(x_{|U'})=0$. Thus in every case $f^{U'}_v(x_{|U'})=x_v$ and this proves the claim. 
\end{subproof}

\begin{claim}
If $S'=\{a,b\}$ then $S\in\Fix(G[U])$.
\end{claim}

\begin{subproof}
Suppose that $S'=\{a,b\}$. Let $v\in U$. Since $f_v(x)\leq f^{U}_v(x_{|U})$, if $x_v=1$ then $f^{U}_v(x)=1$. So suppose that $x_v=0$. We consider two cases. Suppose first that $v$ is adjacent to a positive edge, say $uv$. Since $x_v=0$ we have $f_u(x)=0=x_u$, and since $G$ has no positive edge between $U$ and $V-U$, we have $u\in U$ and we deduce that $f^U_v(x_{|U})=0$. Now, suppose that $v$ is not adjacent to a positive edge. Then $G$ has a negative edge $uv$ with $x_u=1$. If $u\in S'=\{a,b\}$ then $v\in X$ since $v$ is adjacent to no positive edge, a contradiction. Therefore $u\in S\subseteq U$ and we deduce that $f^U_v(x_{|U})=0$. Thus $f^U_v(x_{|U})=x_v$ in every case, and this proves the claim. 
\end{subproof}

\begin{claim}
$S\cup S'\in\Fix(G-\{a,b\})$. 
\end{claim}

\begin{subproof}
Since every vertex in $U'-\{a,b\}$ is adjacent to $a$ or $b$ by a negative edge, if $\{a,b\}\subseteq S'$ then $S'=\{a,b\}$, which is not possible by Claims 1 and 2. We deduce that $S$ and $\{a,b\}$ are disjoint. By Lemma~\ref{lemma:decompo4} we have   
\[
S\cup S'\in \Fix(G)=\Fix(G-\{a,b\}) \cup \big(\{a,b\}\sqcup\Fix(G-\N_G(\{a,b\}))\big).
\]
and the claim follows. 
\end{subproof}

We are now in position to prove the lemma. We have already proved that 
\begin{equation}\label{pro:edge2.5}
\Fix(G)-\Fix(G,U)\subseteq \Fix(G-\{a,b\}).
\end{equation}
Since $G'-\{a,b\}=G-\{a,b\}$ we have
\begin{equation}\label{pro:edge2.6}
\Fix(G',V-\{a,b\})=\{S\cup S'\,|\,S\in\Fix(G-\{a,b\}),~S'\in\Fix(G'[\{a,b\}]-\N_{G'}(S))\}. 
\end{equation}
By Lemma~\ref{lemma:decompo1} we have $\Fix(G',V-\{a,b\})\subseteq \Fix(G')$, and we deduce from \eqref{pro:edge2.5} and \eqref{pro:edge2.6} that there exists a maps $x\mapsto x'$ from $\Fix(G)-\Fix(G,U)$ to $\Fix(G')$ such that 
\[
\ONE(x)\subseteq \ONE(x')\subseteq \ONE(x)\cup\{a,b\}.
\]
Thus $x\mapsto x'$ is an injection and by Claim 1, we have 
\[
\ONE(x')\cap U=\ONE(x)\cap U\not\in\Fix(G[U])=\Fix(G'[U]).
\]
So $x'\not\in\Fix(G',U)$ and we deduce that $x\mapsto x'$ is an injection from $\Fix(G)-\Fix(G,U)$ to $\Fix(G')-\Fix(G',U)$. Thus $|\Fix(G)-\Fix(G,U)|\leq |\Fix(G')-\Fix(G',U)|$. By Lemma~\ref{lemma:decompo1} we have $\Fix(G,U)\subseteq \Fix(G)$ and $\Fix(G',U)\subseteq \Fix(G')$ and the lemma follows.
\end{proof}

Figure~\ref{fig:edge2_2} is an example illustrating the previous lemma. 

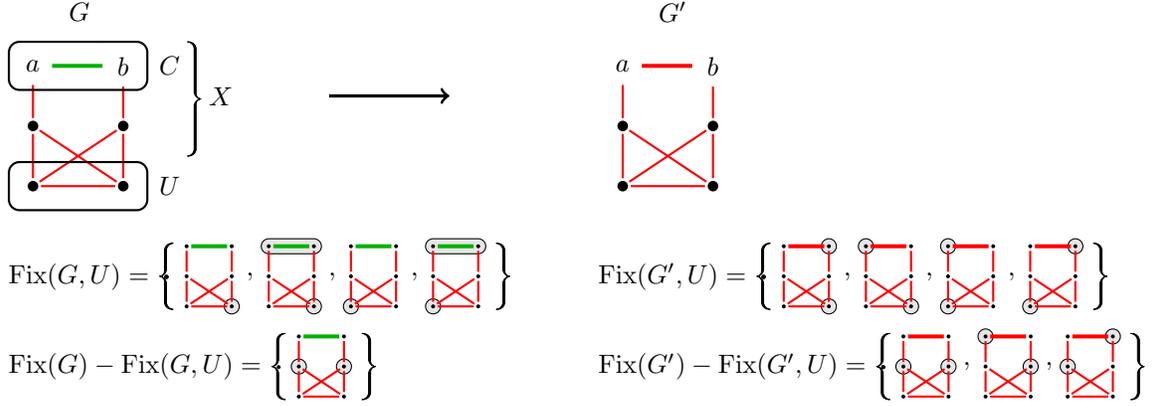
\begin{figure}[!htb]
\centering
\def\desp{9.8\unit}
\newcommand{\grafo}[2][]{ \node[inode,texto,#1] (u) at (cu){$a$};
 \node[inode,texto,#1] (v) at (cv){$b$};
 \foreach \x in{1,...,4}
  \node[vnode,#1] (\x) at (c\x){};
 \path[dneg]
  (1) edge (u)
  (1) edge (3)
  (1) edge (4)
  (2) edge (v)
  (2) edge (4)
  (2) edge (3)
  (3) edge (4)
  (u) edge[#2,grueso] (v);
}
\newcommand{\grafoC}[2][]{%
 \foreach \x in{1,2,3,4,5,u,v}
  \node[vnode,minimum size=1.5pt,#1] (\x) at (c\x){};
 \path[dneg]
  (1) edge (u)
  (1) edge (3)
  (1) edge (4)
  (2) edge (v)
  (2) edge (3)
  (2) edge (4)
  (3) edge (4)
  (u) edge[#2,grueso] (v);
}
 \newcommand{\coordenadas}[2]{
 #1;
 \coordinate[shift={(1.5*#2,0)}] (c2) at (c1);
 \coordinate[shift={(0,-#2)}] (c3) at (c1);
 \coordinate[shift={(1.5*#2,-#2)}] (c4) at (c1);
 \coordinate[shift={(0,#2)}] (cu) at (c1);
 \coordinate[shift={(1.5*#2,#2)}] (cv) at (c1);}

 \begin{tikzpicture}
 \coordenadas{\coordinate (c1) at (0,0)}{\unit}
 \grafo{dpos}
  \node[texto,right] (G) at (0.5\unit,1.9\unit) {$G$};
  \node[texto,right] (G') at (10.3\unit,1.9\unit) {$G'$};
  \draw[cuadros] (u)++(-0.4\unit,-0.4\unit) rectangle +(2.3\unit,0.8\unit);
  \draw[cuadros] (3)++(-0.4\unit,-0.4\unit) rectangle +(2.3\unit,0.8\unit);
  \node[texto,right] (lg) at (2\unit,\unit) {$C$};
  \node[texto,right] (llg) at (2\unit,0) {};
  \node[texto,right] (lg) at (2\unit,-\unit) {$U$};
   \llaveD{($(llg.0)-(-0.3,0.5\unit)$)}{($(llg.0)+(0.3,1.4\unit)$)}
  \node[texto,shift={(3.5\sep,0)}] (lg) at (2.6\unit,0.5\unit){$X$};

  
 \node[texto,right](l) at (-0.5\unit,-2.5\unit) {$\Fix(G,U)=$};
  \llaveI{($(l.0)+(2\sep,-1.1\punit)$)}{($(l.0)+(2\sep,1.1\punit)$)}
 \coordenadas{\coordinate[xshift=3.5\sep] (c1) at (l.0)}{\punit}
 \filldraw[claro] (c4) circle(\ssombra);
 \grafoC{dpos}
 
 \node[shift={(\sep,0)},texto,right](l) at (2) {$,$};
 
 \coordenadas{\coordinate[shift={(\sep,0)}] (c1) at (l.0)}{\punit}
 \filldraw[claro] (c4) circle(\ssombra) (cu)++(90:\ssombra) arc(90:270:\ssombra)to ++(1.5\punit,0) arc(-90:90:\ssombra)--cycle;
 \grafoC{dpos}
 
 \node[shift={(\sep,0)},texto,right](l)at (2) {$,$};
 
 \coordenadas{\coordinate[shift={(\sep,0)}] (c1) at (l.0)}{\punit}
 \filldraw[claro] (c3) circle(\ssombra);
 \grafoC{dpos}
 
 \node[shift={(\sep,0)},texto,right](l)at (2) {$,$};
 
 \coordenadas{\coordinate[shift={(\sep,0)}] (c1) at (l.0)}{\punit}
 \ \filldraw[claro] (c3) circle(\ssombra) (cu)++(90:\ssombra) arc(90:270:\ssombra)to ++(1.5\punit,0) arc(-90:90:\ssombra)--cycle;
 \grafoC{dpos}
  \llaveD{($(2)+(2\sep,-1.1\punit)$)}{($(2)+(2\sep,1.1\punit)$)}
  

 \node[texto,right](l) at (-0.5\unit,-4\unit) {$\Fix(G)-\Fix(G,U)=$};
 \llaveI{($(l.0)+(2\sep,-1.1\punit)$)}{($(l.0)+(2\sep,1.1\punit)$)}
 \coordenadas{\coordinate[xshift=3.5\sep] (c1) at (l.0)}{\punit}
 \filldraw[claro] (c1) circle(\ssombra) (c2)circle(\ssombra);
 \grafoC{dpos}
   \llaveD{($(2)+(2\sep,-1.1\punit)$)}{($(2)+(2\sep,1.1\punit)$)}
 
 \flecha{(lg.0)++(1.5\unit,0)}{2\unit}
 
 \pgftransformxshift{\desp}
 \coordenadas{\coordinate (c1) at (0,0)}{\unit}
  \grafo{}

 \node[texto,right](l) at (-0.5\unit,-2.5\unit) {$\Fix(G',U)=$};
 \llaveI{($(l.0)+(2*\sep,-1.1\punit)$)}{($(l.0)+(2*\sep,1.1\punit)$)}
  \coordenadas{\coordinate[xshift=3.5\sep] (c1) at (l.0)}{\punit}
 \filldraw[claro] (cv) circle(\ssombra) (c4)circle(\ssombra);
 \grafoC{}
 
 \node[shift={(\sep,0)},texto,right](l) at (2) {$,$};
 
 \coordenadas{\coordinate[shift={(\sep,0)}] (c1) at (l.0)}{\punit}
 \filldraw[claro] (cu) circle(\ssombra) (c4)circle(\ssombra);
 \grafoC{}
 
 \node[shift={(\sep,0)},texto,right](l)at (2) {$,$};
 
 \coordenadas{\coordinate[shift={(\sep,0)}] (c1) at (l.0)}{\punit}
  \filldraw[claro] (cu) circle(\ssombra) (c3)circle(\ssombra);
  \grafoC{}
 
 \node[shift={(\sep,0)},texto,right](l)at (2) {$,$};
 \coordenadas{\coordinate[shift={(\sep,0)}] (c1) at (l.0)}{\punit}
 \filldraw[claro] (cv) circle(\ssombra) (c3)circle(\ssombra);
 \grafoC{}
  \llaveD{($(2)+(2\sep,-1.1\punit)$)}{($(2)+(2\sep,1.1\punit)$)}


  \node[texto,right](l) at (-0.5\unit,-4\unit) {$\Fix(G')-\Fix(G',U)=$};
  \llaveI{($(l.0)+(2*\sep,-1.1\punit)$)}{($(l.0)+(2*\sep,1.1\punit)$)}
 \coordenadas{\coordinate[xshift=3.5\sep] (c1) at (l.0)}{\punit}
 \filldraw[claro] (c1) circle(\ssombra) (c2)circle(\ssombra);
 \grafoC{}
 
 \node[shift={(\sep,0)},texto,right](l) at (2) {$,$};
 
 \coordenadas{\coordinate[shift={(\sep,0)}] (c1) at (l.0)}{\punit}
 \filldraw[claro] (cu) circle(\ssombra) (c2)circle(\ssombra);
 \grafoC{}
 
 \node[shift={(\sep,0)},texto,right](l)at (2) {$,$};
 \coordenadas{\coordinate[shift={(\sep,0)}] (c1) at (l.0)}{\punit}
 \filldraw[claro] (cv) circle(\ssombra) (c1)circle(\ssombra);
 \grafoC{}
  \llaveD{($(2)+(2\sep,-1.1\punit)$)}{($(2)+(2\sep,1.1\punit)$)}

\end{tikzpicture}
 \caption{Illustrative example of Lemma~\ref{lemma:edge2} with strict inequality.}\label{fig:edge2_2}
\end{figure}

We are now in position to conclude the proof with the following quantitative version of Lemma~\ref{lemma:C_4}.

\begin{lemma}[Quantitative version of Lemma~\ref{lemma:C_4}]\label{lemma:edge3}
Suppose that $G$ has a positive edge $ab$ such that $a$ and $b$ are adjacent to no other positive edge. Let $X$ be the set of vertices containing $a$, $b$ and all the vertices of $N_G(a)\cup N_G(b)$ adjacent to only negative edges. Let $\Omega$ be the set of $S\in\Fix(G-X)$ such that $G[X]-\N_G(S)$ has an induced copy of $C_4$ containing $ab$. Let $G'$ be the simple signed graph obtained from $G$ by making $ab$ negative. Then   
\[
\fix(G)\leq \fix(G')+|\Omega|\leq\frac{3}{2}\fix(G').
\]
\end{lemma}

\begin{remark}
If $\fix(G)>\fix(G')$ then $\Omega$ is not empty and we deduce that $G[X]$ has an induced copy $C$ of $C_4$ containing $ab$. By the definition of $X$, $a$ and $b$ are the only vertices of $C$ adjacent to a positive edge. This is why this lemma implies Lemma~\ref{lemma:C_4}. 
\end{remark}

\begin{proof}[Proof of Lemma~\ref{lemma:edge3}]
Let $U=V-X$ and for every $S\in\Fix(G[U])$, let $G_S=G[X]-\N_G(S)$ and $G'_S=G'[X]-\N_{G'}(S)$. Since $G[U]=G'[U]$ we have  
\[
\begin{array}{lll}
\Fix(G,U)&=&\{S\cup S'\,|\,S\in\Fix(G[U]),~S'\in \Fix(G_S)\}\\[1mm]
\Fix(G',U)&=&\{S\cup S'\,|\,S\in\Fix(G[U]),~S'\in \Fix(G'_S)\}.
\end{array}
\]
Thus
\[
\fix(G,U)-\fix(G',U)=\sum_{S\in\Fix(G[U])}\fix(G_S)-\fix(G'_S).
\]
For every $S\in\Fix(G[U])$ we have $X\cap \N_G(S)=X\cap\N_{G'}(S)$ so $G'_S$ is obtained from $G_S$ by making $ab$ negative. Hence, according to Lemma~\ref{lemma:edge1}, we have 
\[
\begin{array}{l}
\forall S\in\Omega,\qquad \fix(G_S)-\fix(G'_S)\leq 1\\[1mm]
\forall S\notin\Omega,\qquad \fix(G_S)-\fix(G'_S)\leq 0. 
\end{array}
\]
Thus 
\[
\fix(G,U)-\fix(G',U)\leq |\Omega|,
\]
and using Lemma~\ref{lemma:edge2} we obtain 
\[
\fix(G)-\fix(G')\leq \fix(G,U)-\fix(G',U)\leq|\Omega|. 
\]
Furthermore, for every $S\in\Omega$, $G'_S$ has only negative edges and contains $ab$. Thus it has a maximal independent set containing $a$, say $S_a$, and a maximal independent set containing $b$, say $S_b$. Then $S\cup S_a$ and $S\cup S_b$ are distinct elements of $\Fix(G',U)$. We deduce that $2|\Omega|\leq\fix(G',U)\leq\fix(G')$, and the proposition follows.   
\end{proof}

\section{Proof of Theorem~\ref{thm:H'}}\label{sec:H'}

Let $G$ be a simple signed graph with vertex $V$ and edge set $E$. Let $C$ be a set of vertices such that $G[C]$ is connected and $|C|\geq 2$. We denote by $G/C$ the simple signed graph obtained from $G$ by contracting $C$ into a single vertex $c$, and by adding a negative edge $cc'$, where $c'$ is a new vertex. Formally: (1) the vertex set of $G/C$ is $V=(V-C)\cup\{c,c'\}$ where $c$ and $c'$ are new vertices; (2) the edge set is $(\{\nu(v)\nu(u)\,|\,uv\in E\}-\{cc\})\cup\{cc'\}$, where $\nu$ is the function that maps every vertex in $V-C$ to itself, and every vertex in $C$ to the new vertex~$c$; and (3) an edge $uv$ of $G/C$ is negative if $u=c$ or $v=c$ and it has the same sign as in $G$ otherwise. 

\begin{lemma}\label{lemma:contraction}
If $C$ is a non-trivial connected component of $G^+$ then 
\[
\fix(G)\leq\fix(G/C), 
\]
and the upper bound is reached if $G[C]$ has no negative edge. 
\end{lemma}

\begin{proof}
Let 
\[
\begin{array}{l}
S_c=\{S\in\Fix(G/C)\,|\,c\in S\}\\[1mm]
S_{\bar c}=\{S\in\Fix(G/C)\,|\,c\not\in S\}.
\end{array}
\]
Since $c$ is adjacent to only negative edges, we have 
\[
S_c=\{c\}\sqcup\Fix\big((G/C)-c-\N_{G/C}(c)\big).
\]
and since there are only negative edges between $C$ and $V-C$ we have 
\[
(G/C)-c-\N_{G/C}(c)=G-\N_{G}(C).
\]
Thus
\[
|S_c|=\fix(G-\N_G(C)).
\]
Now, since $c'$ has $c$ as unique neighbor, and since $cc'$ is negative, we have 
\[
S_{\bar c}=\{S\in\Fix(G/C)\,|\,c'\in S\}=\{c'\}\sqcup\Fix(G/C-\{c,c'\}).
\]
Since $G/C-\{c,c'\}=G-C$ we deduce that 
\[
|S_{\bar c}|=\fix(G-C).
\]
Thus
\[
\fix(G-C)+\fix(G-\N_G(G))=|S_{\bar c}|+|S_c|=\fix(G/C). 
\]
According to  Lemma~\ref{lemma:decompo4}, 
\[
\fix(G)=\fix(G-C)+\fix(G-\N_G(G))=\fix(G/C)
\]
if $G[C]$ has no negative edge, and $\fix(G)=\fix(G-C)\leq \fix(G/C)$ otherwise.
\end{proof}

\begin{proof}[Proof of Theorem~\ref{thm:H'}]
Let $G$ be an unsigned graph. Let $H$ be the member of $\H'(G)$ that maximizes $\mis(H)$. Let $\s$ be a repartition of signs in $G$ that maximizes  $\fix(G_\s)$ and such that the number of positive edges in $G_\s$ is minimal for this property. We want to prove that $\fix(G_\s)=\mis(H)$.
According to Lemmas~\ref{lemma:symmetric} and \ref{lemma:matching}, $G_\s$ is a simple signed graph in which the positive edges form a matching, say $u_1v_1,\dots,u_kv_k$. Let $H^0=G_\s$ and for $1\leq\ell\leq k$, let $H^\ell=H^{\ell-1}/\{u_k,v_k\}$. By Lemma \ref{lemma:contraction} we have 
\[
\fix(G_\s)=\fix(H^0)=\fix(H^1)=\cdots =\fix(H^k)
\]
and since $H^k$ has no positive edges, $\fix(H^k)=\mis(H^k)$. Since the underlying unsigned graph of $H^k$ is a member of $\H'(G)$ this proves that 
\[
\fix(G_\s)=\mis(H^k)\leq \mis(H).
\]

Since $H\in\H'(G)$, there exists disjoint subsets of vertices $C_1,\dots,C_k$ and a sequence of graphs $H^0,\dots,H^k$ with $H^0=G$ and $H^k=H$ such that $H^{\ell-1}[C_\ell]$ is connected and $H^\ell=H^{\ell-1}/C_\ell$ for all $1\leq\ell\leq k$. For $1\leq\ell\leq k$, let $\s_\ell$ be the repartition of signs in $H^\ell$ such that $\s_\ell(uv)$ is positive if and only if $u,v\in C_p$ for some $\ell<p\leq k$. In this way we have $H^{\ell}_{\s_\ell}=H^{\ell-1}_{\s_{\ell-1}}/C_\ell$ for $1\leq \ell\leq k$, and by Lemma \ref{lemma:contraction} we have 
\[
\fix(G_{\s_0})=\fix(H^0_{\s_0})=\fix(H^1_{\s_1})=\cdots =\fix(H^k_{\s_k}).
\]
Since $H^k_{\s_k}$ has only negative edges, we deduce that 
\[
\fix(G_\s)\geq \fix(G_{\s_0})=\mis(H^k)=\mis(H).
\]
This proves that $\fix(G_\s)=\mis(H)$.
\end{proof}

\section{Proof of Theorems~\ref{thm:NP-hard} and \ref{thm:NP-hard2}}\label{sec:NP-Hard}

Let us begin with an easy complexity result, proved with a straightforward reduction to SAT similar to the one introduced in \cite{JYP88}. 

\begin{proposition}\label{pro:U}
Given a graph $G$ and a subset $U$ of its vertices, it is NP-hard to decide if $G$ has a maximal independent set disjoint from $U$. 
\end{proposition}

\begin{proof}
Let $\phi$ be a CNF-formula with variables $x_1,\dots,x_n$ and clauses $C_1,\dots,C_k$. Let $G$ be the graph defined as follows. The vertices of $G$ are the positive literals $x_1,\dots,x_n$, the negative literals $\overline{x_1},\dots,\overline{x_n}$, and the clause $C_1,\dots,C_k$. The edges are defined as follows: there is an edge connecting any two contradict literal $x_i$ and $\overline{x_i}$, and each clause $C_i$ is adjacent to all literals it contains. It is then clear that $\phi$ is satisfiable if and only if $G$ has a maximal independent set disjoint from the set of clauses. 
\end{proof}

According to this proposition, the following lemma is a good step toward Theorems~\ref{thm:NP-hard} and \ref{thm:NP-hard2}. 

\begin{lemma}\label{lem:tildeGandU}
Let $G$ be a graph and let $U$ a non-empty subset of vertices. Let $\tilde G$ be the graph obtained from $G$ by adding four additional vertices $a,b,c,d$, the edges $ab$, $bc$, $cd$, $da$, and an edge $av$ for every vertex $v\in U$. Suppose that $\tilde G$ has a unique induced copy of $C_4$, the one induced by $\{a,b,c,d\}$. Then the following are equivalent:
\begin{enumerate}
\item 
$\max_{\s} \fix(\tilde G_\s)=\mis(\tilde G)$,
\item 
$\max_{H\in \H(\tilde G)}=\mis(\tilde G)$,
\item
$G$ has no maximal independent set disjoint from $U$.
\end{enumerate} 
\end{lemma}

\begin{proof}
\setcounter{claim}{0}
We first need the following claim. 

\begin{claim}\label{claim:misG}
We have 
\[
\mis(\tilde G)\leq 2\mis(G)+\mis(G-U)
\] 
and the bound is reached if and only if $G$ has no maximal independent set disjoint from $U$. 
\end{claim}     

\begin{subproof}
Let 
\[
\begin{array}{l@{\,=\,}l}
M_{\bar c}&\{S\in\Mis(\tilde G)\,|\,c\not\in S\}\\ 
M_{ca}&\{S\in\Mis(\tilde G)\,|\,c\in S,~a\in S\}\\
M_{c\bar a}&\{S\in\Mis(\tilde G)\,|\,c\in S,~a\not\in S\}\\
\end{array}
\]
Clearly, these three sets form a partition of $\Mis(G)$, and it is easy to see that 
\[
\begin{array}{l@{\,=\,}l}
M_{\bar c}&\{b,d\}\sqcup\Mis(G)\\ 
M_{ca}&\{c,a\}\sqcup \Mis(G-U)\\
M_{c\bar a}&\{c\}\sqcup \{S\in\Mis(G)\,|\,S\cap U\neq\emptyset\}\\
\end{array}
\]
So $|M_{\bar c}|=\mis(G)$ and $|M_{ca}|=\mis(G-U)$ and $|M_{c\bar a}|\leq\mis(G)$. Since $|M_{c\bar a}|=\mis(G)$ if and only if $G$ has no maximal independent set disjoint from $U$, the claim is proved. 
\end{subproof}

Let $\tilde G_{ab}$ be the simple signed graph obtained from $\tilde G$ by labeling $ab$ with a positive sign and all the other edges by a negative sign. Let $\tilde G_{bc}$, $\tilde G_{cd}$ and $\tilde G_{da}$ be defined similarly. 

\begin{claim}\label{claim:fixGab}
We have  
\[
\fix(\tilde G_{ab})=
\fix(\tilde G_{bc})=
\fix(\tilde G_{cd})=
\fix(\tilde G_{da})=2\mis(G)+\mis(G-U).
\]
\end{claim} 

\begin{subproof}
Since $\tilde G_{ab}$ has $ab$ as unique positive edge, we deduce from Lemma~\ref{lemma:decompo4} that
\[
\fix(\tilde G_{ab})=\mis(G_1)+\mis(G_2)
\quad\text{with}\quad
\left\{
\begin{array}{l}
G_1=\tilde G-\{a,b\}\\[1mm]
G_2=\tilde G-N_{\tilde G}(\{a,b\}).
\end{array}
\right.
\]
Since $G_1$ is the disjoint union of the edge $cd$ and $G$, we have $\mis(G_1)=2\mis(G)$. Furthermore, since $N_{\tilde G}(\{a,b\})=\{a,b,c,d\}\cup U$, we have $G_2=G-U$. So $\mis(\tilde G_{ab})=2\mis(G)+\mis(G-U)$, and by symmetry $\mis(\tilde G_{da})=\mis(\tilde G_{ab})$. 
 
Similarly, since $\tilde G_{bc}$ has $bc$ as unique positive edge, we deduce from Lemma~\ref{lemma:decompo4} that
\[
\fix(\tilde G_{bc})=\mis(G_1)+\mis(G_2)
\quad\text{with}\quad
\left\{
\begin{array}{l}
G_1=\tilde G-\{b,c\}\\[1mm]
G_2=\tilde G-N_{\tilde G}(\{b,c\}).
\end{array}
\right.
\]
In $G_1$, $d$ has $a$ as unique neighbor. Thus, the number of maximal independent sets of $G_1$ containing $d$ is $\mis(G_1-\{a,d\})=\mis(G)$; and the number of maximal independent sets of $G_1$ not containing $d$ is equal to the number of maximal independent sets containing $a$, which is $\mis(G_1-a-U-d)=\mis(G-U)$.
Since $N_{\tilde G}(\{b,c\})=\{a,b,c,d\}$, we have $G_2=G-U$ thus $\mis(\tilde G_{bc})=2\mis(G)+\mis(G-U)$. By symmetry $\mis(\tilde G_{cd})=\mis(\tilde G_{bc})$.
\end{subproof}

Let $\s$ be a repartition of sign in $\tilde G$ that maximizes $\fix(\tilde G_\s)$, and such that the number of positive edges in $\tilde G_\s$ is minimal for this property. If $\fix(\tilde G_\s)=\mis(\tilde G)$ then $\fix(\tilde G_{ab})\leq \mis(\tilde G)$ and we deduce from Claims~\ref{claim:misG} and \ref{claim:fixGab} that $G$ has no maximal independent set disjoint from $U$. Otherwise, $\fix(\tilde G_\s)>\mis(\tilde G)$ and since $\tilde G$ has a unique induced copy of $C_4$  and we deduce from Lemmas~\ref{lemma:C_4} that $\tilde G_\s$ has a unique positive edge $e\in\{ab,bc,cd,da\}$. Thus by Claim~\ref{claim:fixGab} we have 
\[
\mis(\tilde G)<\fix(\tilde G_\s)=2\mis(G)+\mis(G-U)
\]  
and we deduce from Claim~\ref{claim:misG} that $G$ has a maximal independent set disjoint from $U$. This shows the equivalence between the point 1 and 3 in the statement. 
 
Let $k=\max_{H\in\H(\tilde G)}\mis(H)$. To conclude the proof, it is sufficient to show that 
\[
\fix(\tilde G_\s)=k.
\]
By Theorem~\ref{thm:H'}, we have $\fix(\tilde G_\s)=\max_{H\in\H'(\tilde G)}\mis(H)\geq k$, so we only need to prove that $\fix(\tilde G_\s)\leq k$. If $\fix(\tilde G_\s)=\mis(\tilde G)$ then $\fix(\tilde G_\s)\leq k$ since $\tilde G\in\H(\tilde G)$. Otherwise, $\fix(\tilde G_\s)>\mis(\tilde G)$ and as above we deduce that $\tilde G_\s$ has a unique positive edge $e\in \{ab,bc,cd,da\}$. Let $H_e$ be the graph obtained from $\tilde G$ by contracting $e$, and let $H'_e=G/e$ be obtain from $G$ by contracting $e$ into a single vertex $c$, and by adding an edge $cc'$. It is easy to check that in every case we have $\mis(H_e)=\mis(H'_e)$, and by Lemma~\ref{lemma:contraction} we have $\mis(H'_e)=\fix(\tilde G_\s)$. Since $H_e\in\H(\tilde G)$ we deduce that $\fix(\tilde G_\s)\leq k$. 
\end{proof}

According to Lemma~\ref{lem:tildeGandU}, to prove Theorems~\ref{thm:NP-hard} and \ref{thm:NP-hard2}, it is sufficient to prove the following strengthening of Proposition~\ref{pro:U}. 

\begin{lemma}\label{lem:GandU}
Let $G$ be a graph and let $U$ be a non-empty subset of its vertices. Let $\tilde G$ be obtained from $G$ and $U$ as in Lemma~\ref{lem:tildeGandU}. It is NP-hard to decide if $G$ has a maximal independent set disjoint from $U$, even if $\tilde G$ has a unique induced copy of $C_4$. 
\end{lemma}

\begin{proof}
\setcounter{claim}{0}
We follow the proof of Proposition \ref{pro:U}, using however a slightly more involved arguments. Let $\phi$ be a CNF-formula with variables $x_1,\dots,x_n$ and clauses $C_1,\dots,C_k$. Let $G$ be the graph defined as follows: the vertex set is the union of the following sets 
\[
\begin{array}{l}
X=\{x_1,\dots,x_n\}\\
\overline{X}=\{\overline{x_1},\dots,\overline{x_n}\}\\
Y=\{y_1,\dots,y_n\}\\
\overline{Y}=\{\overline{y_1},\dots,\overline{y_n}\}\\
\mathcal{L}=\{L_1,\dots,L_n\}\\
\overline{\mathcal{L}}=\{\overline{L_1},\dots,\overline{L_n}\}\\
\mathcal{C}=\{C_1,\dots,C_k\}.
\end{array}
\]
The edge set is defined by: for all $1\leq i\leq n$, $x_iy_i$, $\overline{x_i}\,\overline{y_i}$, $x_iL_i$, $\overline{x_i}L_i$, $y_i\overline{L_i}$, and $\overline{y_i}\overline{L_i}$ are edges; for all $1\leq i<j\leq k$, $C_iC_j$ is an edge; for all $1\leq i\leq n$ and $1\leq j\leq k$, $L_iC_j$ is an edge; for all $1\leq i\leq n$ and $1\leq j\leq k$, $x_iC_j$ is an edge if $x_i$ is a positive literal of the clause $C_j$; and finally, for all $1\leq i\leq n$ and $1\leq j\leq k$, $\overline{x_i}C_j$ is an edge if $\overline{x_i}$ is a negative literal of the clause $C_j$ (see Figure~\ref{fig:Complejidad}). We set 
\[
U=\mathcal{C}\cup\mathcal{L}\cup\overline{\mathcal{L}}.
\]
\begin{figure}[!htb]
\centering
\begin{tikzpicture}
\node[draw, rounded corners, minimum size=2cm,very thick, text width=1.9cm, text centered](C) at (0,0){$\mathcal{C}$ (clique of clauses)};
\def\radio{4}
\node (G) at (125:\radio+1.5) {$G$};
\foreach \x/\y in{130/1,100/2,50/n}{
 \node[inode,shift={(\x:\radio)}] (L\y) at (\x-180:0.75) {$L_{\y}$};
 \node[inode,shift={(\x:\radio)}] (x\y) at (\x-120:0.75) {$x_{\y}$};
 \node[inode,shift={(\x:\radio)}] (y\y) at (\x-60:0.75) {$y_{\y}$};
 \node[inode,shift={(\x:\radio)}] (nx\y) at (\x+120:0.75) {$\overline{x_{\y}}$};
 \node[inode,shift={(\x:\radio)}] (ny\y) at (\x+60:0.75) {$\overline{y_{\y}}$};
 \node[inode,shift={(\x:\radio)}] (nL\y) at (\x:0.75) {$\overline{L_{\y}}$};
 \path[very thick]
 (L\y) edge (x\y)
 (x\y) edge (y\y)
 (y\y) edge (nL\y)
 (nL\y) edge (ny\y)
 (ny\y) edge (nx\y)
 (nx\y) edge (L\y)
 (L\y) edge (C)
 ;
 }
 \path (y2) to node[sloped] {$\dots$} (nyn);
 
\pgftransformshift{\pgfpoint{7cm}{2cm}}
\def\radio{1}
\draw[very thick, rounded corners] (-135:4.3) rectangle (45:4.3) ;
\node[below right] (C) at (135:4.3) {$C$};
\foreach \x/\y in {90/1,162/2,234/3,306/4,18/k}{
 \node[inode](C\y) at (\x:1.75){$C_\y$};
 \node[inode,shift={(\x+75:\radio)}](L1) at (C\y){$L_1$};
 \node[inode,shift={(\x+15:\radio)}](Ln) at (C\y){$L_n$};
 \node[inode,shift={(\x-15:\radio)}](lit1) at (C\y){$l^{\y}_1$};
 \node[inode,shift={(\x-45:\radio)}](lit2) at (C\y){$l^{\y}_2$};
 \node[inode,shift={(\x-75:\radio)}](lit3) at (C\y){$l^{\y}_3$};
 \path (L1) to node[sloped] {$\dots$} (Ln) ;
 \path[very thick]
 (C\y) edge (L1)
 (C\y) edge (Ln)
 (C\y) edge (lit1)
 (C\y) edge (lit2)
 (C\y) edge (lit3)
 ;
 }
\path[very thick]
(C1) edge (C2)
(C1) edge (C3)
(C1) edge (C4)
(C2) edge (C3)
(C2) edge (C4)
(C4) edge (C3)
(Ck) edge (C1)
(Ck) edge (C2)
(Ck) edge (C3)
(Ck)to node[sloped] {$\dots$}(C4)
;
\end{tikzpicture}
\caption{Illustration of $G$ (when each clause $C_i$ contains three literals $l^i_1,l^i_2,l^i_3\in\mathcal{L}\cup\overline{\mathcal{L}}$).}\label{fig:Complejidad}
\end{figure}
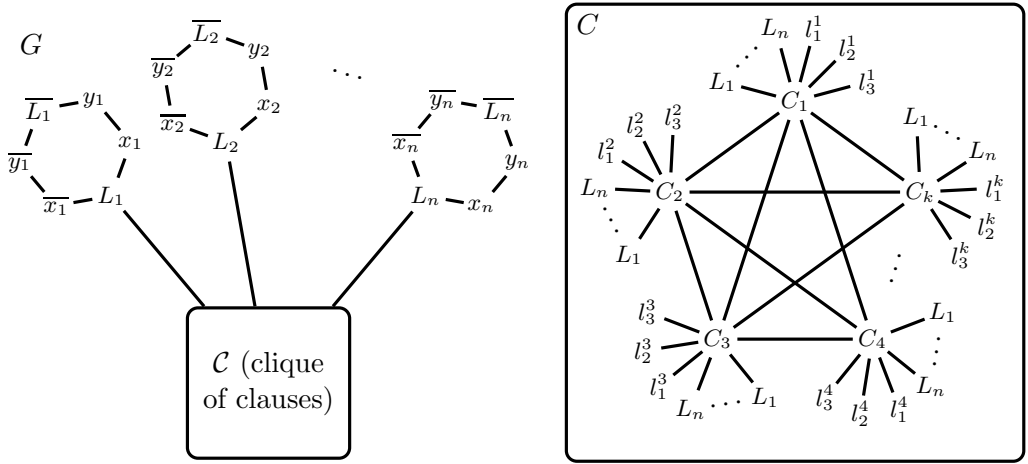

\begin{claim}\label{claim:SAT}
$\phi$ is satisfiable if and only if there exists $S\in\Mis(G)$ such that $S\cap U=\emptyset$.  
\end{claim}     

\begin{subproof}
Suppose that $S\in\Mis(G)$ and $S\cap U=\emptyset$. Consider the assignment defined by $x_i=1$ is $x_i\in S$ and $x_i=0$ otherwise. Let $C_j$ be any clause of $\phi$. Since $S\cap U=\emptyset$, there exists $1\leq i\leq n$ such that $x_iC_j$ is an edge with $x_i\in S$, or such that $\overline{x_i}C_j$ is an edge with $\overline{x_i}\in S$. If $x_iC_j$ is an edge with $x_i\in S$ then $x_i$ is a positive literal of the clause $C_j$ and $x_i=1$ thus the clause $C_j$ is made true by the assignment. Suppose now that $\overline{x_i}C_j$ is an edge with $\overline{x_i}\in S$. Then $\overline{y_i}\not\in S$ and since $\overline{L_i}$ has a neighbor in $S$, we deduce that $y_i\in S$, and consequently, $x_i\not\in S$. Thus $x_i=0$, and since $\overline{x_i}C_j$ is an edge, $\overline{x_i}$ is a negative literal of $C_j$, and thus $C_j$ is made true by the assignment. Hence, every clauses is made true, thus $\phi$ is satisfiable. 
  
Suppose now that there exists a assignment that makes $\phi$ true, and let  
\[
S=
\{x_i,\overline{y_i}\,|\,x_i=1,~1\leq i\leq n\}\cup
\{\overline{x_i},y_i\,|\,x_i=0,~1\leq i\leq n\}
\]
Then $S$ is an independent set disjoint from $U$, and it is maximal if every vertex in $U$ has a neighbor in $S$. For $1\leq i\leq n$, since either $x_i$ or $\overline{x_i}$ is in $S$, the vertex $L_i$ has a neighbor in $S$, and since either $y_i$ or $\overline{y_i}$ is in $S$, the vertex $\overline{L_i}$ has also a neighbor in $S$. Now let $C_j$ be any clause of $\phi$. Since $C_j$ is made true by the assignment, it contains a positive literal $x_i$ with $x_i=1$ or a negative literal $\overline{x_i}$ with $x_i=0$. In the first case, $x_i\in S$ and $x_iC_j$ is an edge, and in the second case, $\overline{x_i}\in S$ and $\overline{x_i}C_j$ is an edge. Thus $S\in\Mis(G)$. 
\end{subproof}

\begin{claim}\label{claim:noC_4}
$G$ has no induced copy of $C_4$.  
\end{claim}     

\begin{subproof}
Suppose, for a contradiction, that $G$ has an induced copy of $C_4$ with vertices $W=\{v_1,v_2,v_3,v_4\}$ given in the order. If $v_1=\overline{L_i}$ for some $1\leq i\leq n$, then $\{v_2,v_4\}=\{y_i,\overline{y_i}\}$ and since $\overline{L_i}$ is the unique common neighbor of $y_i$ and $\overline{y_i}$, there is a contradiction. We deduce that $W\cap\overline{\mathcal{L}}=\emptyset$. Since $\{x_i,\overline{L_i}\}$ is the neighborhood of $y_i$ we deduce that $W\cap Y=\emptyset$, and similarly, $W\cap \overline{Y}=\emptyset$. Hence, 
\[
W\cap (Y\cup\overline{Y}\cup\overline{\mathcal{L}})=\emptyset.
\]
Suppose now that $v_1=L_i$ for some $1\leq i\leq n$. If $v_2=C_j$ for some $1\leq j\leq k$ then $v_3\in X\cup \overline{X}\cup \C\cup\mathcal{L}$ and we deduce that $v_3\in X\cup \overline{X}$ since otherwise $v_1v_3$ is an edge. So $v_4\in\C\cup\mathcal{L}$ and we deduce that $v_2v_4$ is an edge, a contradiction. Thus $v_2\in\{x_i,\overline{x_i}\}$. Suppose that $v_2=x_i$, the other case being similar. Then $v_3\in \C\cup\{y_i\}$, and thus $v_3\in\C$, so $v_1v_3$ is an edge, a contradiction. We deduce that $v_1\not\in\mathcal{L}$, and thus $U\cap\mathcal{L}=\emptyset$. Thus $W\subseteq X\cup \overline{X}\cup \C$. Since $X\cup Y$ is an independent set, and since $G[\C]$ is complete, we easily obtain a contradiction. 
\end{subproof}

Let $\tilde G$ be obtained from $G$ and $U$ as in Lemma \ref{lem:tildeGandU}.

\begin{claim}\label{claim:uniqueC_4}
$\tilde G$ has a unique induced copy of $C_4$, the one induced by $\{a,b,c,d\}$. 
\end{claim}     

\begin{subproof}
Suppose for a contradiction that $\tilde G$ has an induced copy of $C_4$ with vertices $W=\{v_1,v_2,v_3,v_4\}$ given in the order, and suppose that $W\neq\{a,b,c,d\}$. Since $G$ has no induced copy of $C_4$, we deduce that $W\cap \{a,b,c,d\}=\{a\}$. Suppose, without loss of generality, that $a=v_1$. Then $v_2,v_4\in U$, and since two distinct vertices in $\mathcal{L}\cup\overline{\mathcal{L}}$ have no common neighbor in $G$, we deduce that $v_2,v_4\in\C$. But then $v_2v_4$ is an edge of $\tilde G$, a contradiction. 
\end{subproof}
\end{proof}
 
\bibliographystyle{plain}
\bibliography{BIB}

\end{document}